\documentclass[12pt]{amsart}
\usepackage{amsmath, amscd}
\usepackage{amsfonts, amssymb, mathtools}
\usepackage{xcolor}

\begin{document}
\numberwithin{equation}{section}

\def\1#1{\overline{#1}}
\def\2#1{\widetilde{#1}}
\def\3#1{\widehat{#1}}
\def\4#1{\mathbb{#1}}
\def\5#1{\frak{#1}}
\def\6#1{{\mathcal{#1}}}

\newcommand{\w}{\omega}
\newcommand{\Lie}[1]{\ensuremath{\mathfrak{#1}}}
\newcommand{\LieL}{\Lie{l}}
\newcommand{\LieH}{\Lie{h}}
\newcommand{\LieG}{\Lie{g}}
\newcommand{\de}{\partial}
\newcommand{\R}{\mathbb R}
\newcommand{\FH}{{\sf Fix}(H_p)}
\newcommand{\al}{\alpha}
\newcommand{\tr}{\widetilde{\rho}}
\newcommand{\tz}{\widetilde{\zeta}}
\newcommand{\tk}{\widetilde{C}}
\newcommand{\tv}{\widetilde{\varphi}}
\newcommand{\hv}{\hat{\varphi}}
\newcommand{\tu}{\tilde{u}}
\newcommand{\tF}{\tilde{F}}
\newcommand{\debar}{\overline{\de}}
\newcommand{\Z}{\mathbb Z}
\newcommand{\C}{\mathbb C}
\newcommand{\Po}{\mathbb P}
\newcommand{\zbar}{\overline{z}}
\newcommand{\G}{\mathcal{G}}
\newcommand{\So}{\mathcal{S}}
\newcommand{\Ko}{\mathcal{K}}
\newcommand{\U}{\mathcal{U}}
\newcommand{\B}{\mathbb B}
\newcommand{\NC}{\mathcal N\mathcal C}
\newcommand{\oB}{\overline{\mathbb B}}
\newcommand{\Cur}{\mathcal D}
\newcommand{\Dis}{\mathcal Dis}
\newcommand{\Levi}{\mathcal L}
\newcommand{\SP}{\mathcal SP}
\newcommand{\Sp}{\mathcal Q}
\newcommand{\A}{\mathcal O^{k+\alpha}(\overline{\mathbb D},\C^n)}
\newcommand{\CA}{\mathcal C^{k+\alpha}(\de{\mathbb D},\C^n)}
\newcommand{\Ma}{\mathcal M}
\newcommand{\Ac}{\mathcal O^{k+\alpha}(\overline{\mathbb D},\C^{n}\times\C^{n-1})}
\newcommand{\Acc}{\mathcal O^{k-1+\alpha}(\overline{\mathbb D},\C)}
\newcommand{\Acr}{\mathcal O^{k+\alpha}(\overline{\mathbb D},\R^{n})}
\newcommand{\Co}{\mathcal C}
\newcommand{\Hol}{{\sf Hol}(\mathbb H, \mathbb C)}
\newcommand{\Aut}{{\sf Aut}(\mathbb D)}
\newcommand{\D}{\mathbb D}
\newcommand{\oD}{\overline{\mathbb D}}
\newcommand{\oX}{\overline{X}}
\newcommand{\loc}{L^1_{\rm{loc}}}
\newcommand{\la}{\langle}
\newcommand{\ra}{\rangle}
\newcommand{\thh}{\tilde{h}}
\newcommand{\N}{\mathbb N}
\newcommand{\kd}{\kappa_D}
\newcommand{\Ha}{\mathbb H}
\newcommand{\ps}{{\sf Psh}}
\newcommand{\Hess}{{\sf Hess}}
\newcommand{\subh}{{\sf subh}}
\newcommand{\harm}{{\sf harm}}
\newcommand{\ph}{{\sf Ph}}
\newcommand{\tl}{\tilde{\lambda}}
\newcommand{\gdot}{\stackrel{\cdot}{g}}
\newcommand{\gddot}{\stackrel{\cdot\cdot}{g}}
\newcommand{\fdot}{\stackrel{\cdot}{f}}
\newcommand{\fddot}{\stackrel{\cdot\cdot}{f}}
\def\v{\varphi}
\def\Re{{\sf Re}\,}
\def\Im{{\sf Im}\,}
\def\rk{{\rm rank\,}}
\def\rg{{\sf rg}\,}
\def\Gen{{\sf Gen}(\D)}
\def\Pl{\mathcal P}

\newtheorem{theorem}{Theorem}[section]
\newtheorem{lemma}[theorem]{Lemma}
\newtheorem{proposition}[theorem]{Proposition}
\newtheorem{corollary}[theorem]{Corollary}

\theoremstyle{definition}
\newtheorem{definition}[theorem]{Definition}
\newtheorem{example}[theorem]{Example}

\theoremstyle{remark}
\newtheorem{remark}[theorem]{Remark}
\numberwithin{equation}{section}



\def\id{{\sf id}}
\newcommand{\Real}{\mathbb{R}}
\newcommand{\Natural}{\mathbb{N}}
\newcommand{\Complex}{\mathbb{C}}
\newcommand{\ComplexE}{\overline{\mathbb{C}}}
\newcommand{\Int}{\mathbb{Z}}
\newcommand{\UD}{\mathbb{D}}
\newcommand{\clS}{\mathcal{S}}
\newcommand{\gtz}{\ge0}
\newcommand{\gt}{\ge}
\newcommand{\lt}{\le}
\newcommand{\fami}[1]{(#1_{s,t})}
\newcommand{\famc}[1]{(#1_t)}
\newcommand{\ts}{t\gt s\gtz}
\newcommand{\classCC}{\tilde{\mathcal C}}
\newcommand{\classS}{\mathcal S}

\newcommand{\step}[2]{\vskip-3mm\hbox{~}\begin{itemize}\item[{\it Step~#1:}]{~\it #2.}\end{itemize}}
\newcommand{\Case}[2]{\vskip-3mm\hbox{~}\begin{itemize}\item[{\it Case~#1:}]{~\it #2.}\end{itemize}}
\newcommand{\proofbox}{\hfill$\Box$}
\newcommand{\Claim}[2]{\vskip-3mm\hbox{~}\begin{itemize}\item[{\it Claim~#1:}]{~\it #2.}\end{itemize}}

\newcommand{\mcite}[1]{\csname b@#1\endcsname}
\newcommand{\UC}{\de\D}

\newcommand{\Moeb}{\mathrm{M\ddot ob}}

\newcommand{\dAlg}{{\mathcal A}(\UD)}
\newcommand{\diam}{\mathrm{diam}}

\renewcommand{\Hol}{{\sf Hol}}

\long\def\REM#1{\relax}

\def\mydot{\stackrel{\text{\Huge.}}}

\newcommand{\Maponto}
{\xrightarrow{\hbox{\lower.2ex\hbox{$\scriptstyle \smash{\mathsf{onto}}$}}\,}}
\newcommand{\Mapinto}
{\xrightarrow{\hbox{\lower.2ex\hbox{$\scriptstyle \smash{\mathsf{into}}$}}\,}}
\newcommand{\anglim}{\angle\lim}

\newenvironment{mylist}{\begin{list}{}%
{\labelwidth=2em\leftmargin=\labelwidth\itemsep=.4ex plus.1ex
minus.1ex\topsep=.7ex plus.3ex
minus.2ex}%
\let\itm=\item\def\item[##1]{\itm[{\rm ##1}]}}{\end{list}}



\title[Contact points and fractional singularities]{Contact points and fractional singularities for semigroups of holomorphic self-maps in the unit disc}
\author[F. Bracci]{Filippo Bracci$^\ast$}
\author[P. Gumenyuk]{Pavel Gumenyuk$^\S$}
\address{Dipartimento Di Matematica\\
Universit\`{a} di Roma \textquotedblleft Tor Vergata\textquotedblright\ \\
Via Della Ricerca Scientifica 1, 00133 \\
Roma, Italy} \email{fbracci@mat.uniroma2.it, gumenyuk@mat.uniroma2.it}

\thanks{$^{*}$Partially supported by the ERC grant ``HEVO - Holomorphic Evolution Equations'' n. 277691.}
\thanks{$^\S$Partially supported by the FIRB grant Futuro in Ricerca ``Geometria Differenziale Complessa e Dinamica Olomorfa'' n. RBFR08B2HY}

\date\today

\begin{abstract} We study boundary singularities which can appear  for infinitesimal generators of one-parameter semigroups of holomorphic self-maps in the unit disc. We introduce ``regular'' fractional singularities and characterize them in terms of the behavior of the associated semigroups and K\oe nigs functions. We also provide necessary and sufficient geometric criteria on the shape of the image of the  K\oe nigs function for having such singularities. In order to do this, we study contact points of semigroups and prove that any contact (not fixed) point  of a one-parameter semigroup corresponds to a maximal  arc on the boundary to which the associated infinitesimal generator extends holomorphically as a vector field tangent to this arc.
\end{abstract}

\maketitle
\tableofcontents

\section{Introduction}

One-parameter semigroups of holomorphic self-maps arise in various areas of analysis, including geometric function theory, operator theory, iteration theory, theory of branching stochastic processes. The study of one-parameter semigroups can be also regarded as a special but very important case of the more general Loewner theory, which gained recently much attention, partially due to the successful development of its stochastic counterpart introduced in~2000 by Schramm~\cite{Schramm}. More about the theory of one-parameter semigroups, its applications and history of the topic can be found in the survey paper~\cite{Goryainov_survey}.

A (continuous) \textsl{one-parameter semigroup} $(\phi_t)$ of holomorphic maps in the unit disc ${\D\subset \C}$ is a continuous homomorphism between the semigroup of non-negative real numbers endowed with the Euclidean topology and the semigroup of all holomorphic self-maps of $\D$ endowed with the topology of uniform convergence on compacta. It is well known (see Section~\ref{SS_semig}) that to each such a one-parameter semigroup one can associate a unique holomorphic vector field $G:\D \to \C$ called the \textsl{infinitesimal generator} of~$(\phi_t)$ whose flow at a point $z\in \D$ is given by $[0,+\infty)\ni t\mapsto \phi_t(z)$. Moreover, there exists an (essentially unique when suitably normalized) univalent function $h:\D\to \C$, called the {\sl K\oe nigs function} of~$(\phi_t)$, which simultaneously linearizes the semigroup. Dynamical properties of~$(\phi_t)$ are related to analytical properties of $G$ and to geometrical properties of the simply connected domain~$h(\D)$.

Since one-parameter groups of automorphisms of  $\D$ are well understood, we assume from now on that $\phi_t$ is not an automorphism for all $t>0$. It is well known that there exists a unique point $\tau\in \oD$, called the {\sl Denjoy\,--\,Wolff point} of $(\phi_t)$, such that $\phi_t(z)\to \tau$ as $t\to +\infty$. The Denjoy\,--\,Wolff point is also a (common) fixed point of $(\phi_t)$, which in case $\tau\in \de \D$ should be understood in the sense of non-tangential limits (see Section~\ref{SS_boundaryreg}). If $\tau\in \D$, then clearly $G(\tau)=0$, and similarly $\angle\lim_{z\to \tau}G(z)=0$ if $\tau\in \de \D$. Moreover, in case $\tau\in \de \D$, also $G'$ has (non-positive) angular limit at~$\tau$, $\anglim_{z\to\tau} h(z)=\infty$ and the width of the minimal horizontal strip containing $h(\UD)$ (which can be infinite) is related to $G'(\tau):=\anglim_{z\to\tau} G'(z)$.  More generally, it is known that  $x\in \de \D$ is a boundary {\sl regular} fixed point for $(\phi_t)$ (\textit{i.e.}, $\phi_t(x)=x$ in the sense of non-tangential limits and $\phi_t$ possesses finite angular derivative at $x$  for all $t\geq 0$) if and only if $x$ is a \textsl{regular null point} of~$G$ (\textit{i.e.}, if $G$ has vanishing non-tangential limit at~$x$ and possesses finite (real) angular derivative at~$x$). In case $x\neq\tau$, this in turn is equivalent to the following condition: $\anglim_{z\to\tau}h(z)=\infty$ and  there is a strip (in case $\tau\in\de\UD$) or a (spiral) sector (in case $\tau\in\UD$) contained in~$h(\UD)$ and containing $h(rx)$  for all~$r\in (0,1)$, see \cite[Theorem~2.6]{AnalyticFlows} and \cite[Lemma~5]{ESZ}. The angle of the maximal sector or the height of the maximal strip having these properties and the angular derivatives $\phi'_t(x)$ and~$G'(x)$ are related to each other in one-to-one manner. It should be noticed that in general non-regular boundary fixed points, the so-called {\sl super-repulsive fixed points}, do not correspond to zeros of the infinitesimal generators, nor the converse is true (see Section~\ref{super S}).

In \cite{BCD-M_regular-poles}, the concept of {\sl regular poles} for an infinitesimal generator was introduced and related to the so-called {\sl $\beta$-points} of $(\phi_t)$ and $h$. The $\beta$-points by themselves are related to the Carleson\,--\,Makarov $\beta$-numbers \cite{CM} introduced in the study of the Brennan conjecture. In \cite{BCD-M_regular-poles} it was proved  that an isolated tip of a radial slit or a ``very thin'' cusp in $h(\D)$  corresponds to a regular pole of~$G$. In Proposition \ref{PR_Bertilsson} we give a characterization of regular poles in terms of the geometry of $h(\D)$ using a condition due to Bertilsson \cite{Bertilsson}.

It is then natural to ask which type of ``regular'' singularities an infinitesimal generator can have on the boundary, and how they are related to the associated semigroup and to the geometry of the corresponding simply connected domain~$h(\UD)$.

Although there are growth estimates restricting possible singularities, the non-tangential boundary behaviour of infinitesimal generators can be quite ``pathological'', see examples in Section~\ref{S_fractional}. In this paper we study rather ``good'' singularities defined as follows: a point $x\in \de \D$ is called a {\sl regular (fractional) singularity of order $\al\in \R\setminus\{0\}$} for $G$ provided $\lim_{r\to 1^-}\frac{G(rx)}{(1-r)^\alpha}$ exists finite and nonzero. An infinitesimal generator can have regular fractional singularities only of order $\al\in [-1,1]\setminus\{0\}$, where $\al=-1$ corresponds to a regular pole and $\al=1$ to a regular null point, except for the case of a regular singularity at the Denjoy\,--\,Wolff point, which was deeply studied in~\cite{Elin-Khavinson-Reich-Shoikhet}. In Section~\ref{S_fractional} we prove the following result:

\begin{theorem}\label{TH_fractional}
Fix $\alpha\in[-1,1)\setminus\{0\}$ and $x\in\UC$. Let $G$ be the infinitesimal generator of a one-parameter semigroup~$(\phi_t)$ and let $h$ be the associated K\oe nigs function. Then the following assertions are equivalent:
\begin{mylist}
\item[(i)] $G$ has a regular singularity of order~$\alpha$ at~$x$;

\item[(ii)] the angular limit
\begin{equation}\label{EQ_fractional-phi-prime}
\anglim_{z\to x}\phi'_{t}(z)(1-\overline x z)^\alpha
\end{equation}
exists and belongs to~$\Complex^*$ for all but at most one~$t>0$;

\item[(iii)] \begin{itemize}
               \item[a)] in case $\al<0$ the angular limit~\eqref{EQ_fractional-phi-prime} exists and belongs to~$\Complex^*$ for at least one~${t>0}$;
               \item[b)] in case $\al>0$ the angular limit~\eqref{EQ_fractional-phi-prime} exists and belongs to~$\Complex^*$ for at least two different positive values of\,~$t$;
             \end{itemize}

\item[(iv)] $\anglim_{z\to x} h'(z)(1-\overline{x}z)^\al$ exists and belongs to~$\Complex^*$;
\item[(v)] $\anglim_{z\to x} (h(z)-h(x))(1-\overline{x}z)^{\al-1}$ exists and belongs to~$\Complex^*$, where $h(x):=\anglim_{z\to x}h(z)$.
\end{mylist}
Moreover, if one (and hence all) of the above assertions hold, then $x$ is not a boundary fixed point of~$(\phi_t)$, and even not a contact point of~$(\phi_t)$ provided $\alpha\in(-1,0)$.
\end{theorem}

The result is sharp, as shown by the examples at the end of Section~\ref{S_fractional}. The proof is based on a preliminary study of  contact points. A point $x\in \de \D$ is a {\sl contact point} for the semigroup $(\phi_t)$ if there exists $t_0>0$ such that the non-tangential limit $\phi_{t_0}(x)\in \de \D$. It turns out that existence of contact points is a quite  rigid condition. We say that an open arc $A\subset \de \D$ is a {\sl contact arc} for the infinitesimal generator $G$ provided $G$ extends holomorphically through $A$ and it is tangent to $A$. In Section~\ref{S contact} we prove the following statement:

\begin{theorem}\label{TH_contact-points}
Let $(\phi_t)$ be a one-parameter semigroup in $\D$  with the associated infinitesimal generator~$G$. Suppose that $x\in \de \D$ is not a boundary fixed point of~$(\phi_t)$. Then the following two assertions are equivalent:
\begin{mylist}
  \item[(i)] there exists $t_0>0$ such that $x$ is a contact point of~$\phi_{t_0}$;
  \item[(ii)] $G$ has a contact arc $A\subset\de \D$, with $x$ being its initial point.
\end{mylist}
If the above equivalent statements hold, then
\begin{mylist}
\item[(iii)] there exists $t_1(x,A)\in(0,+\infty]$, the {\sl life-time of $x$ in $A$}, such that $\phi_t(x)\in A$ for all $t\in(0,t_1(x,A))$ and the map $t\mapsto \phi_t(x)$ is a homeomorphism of $(0,t_1(x,A))$ onto $A$.
\end{mylist}
\end{theorem}

{\sl Maximal} contact arcs are then classified in Proposition~\ref{PR_contact_arcs}. In particular, every element of a one-parameter semigroup can have at most a countable set of non-fixed contact points at which it does not extends holomorphically (see Corollary~\ref{CR_non-embedd}).

In Section~\ref{S geometric} we use the Rodin\,--\,Warschawski theory \cite{RodinWarschawski} in order to give quite general sufficient and necessary geometric criteria on the shape of $h(\D)$ for $G$ to have regular fractional singularities of order $\alpha\in [-1,1)\setminus\{0\}$. The main sufficient criterion for $\alpha\in (-1,1)\setminus\{0\}$ is contained in Theorem \ref{TH_geometr-charact-alpha_neq_1} (see Theorem \ref{TH_geometr-charact-alpha=1} for the case $\alpha=-1$), while the necessary criterion is the content of Theorem \ref{TH_necessary-frac-singularity}. Very roughly speaking, those results say that a point $x\in \de \D$ is a regular fractional singularity of order $\al\in (-1,1)$ if close to $h(x)$ the domain $h(\D)$ looks like  an angle with vertex $h(x)$ and magnitude $(1-\alpha)\pi$.

The plan of the paper is the following. In Section~\ref{S prel} we recall some basic results we need in the paper and prove some other preliminary results. In Section~\ref{S contact} we study contact points of semigroups and contact arcs of infinitesimal generators. In Section~\ref{super S} we discuss super-repulsive fixed points giving an example of a one-parameter semigroup with a super-repulsive fixed point at which the associated infinitesimal generator fails to have radial limit (in \cite[p.\,127]{CDP2} there is an example of an infinitesimal generator having a non-regular null point which does not correspond to a fixed point of the associated one-parameter semigroup). However, we show in Proposition \ref{PR_superrepulsive} that if there is no backward orbits landing at a super-repulsive fixed point then this point is the initial point of a maximal contact arc and in particular it is a (non-regular) null point of the associated infinitesimal generator. In Section~\ref{S_fractional} we define fractional singularities for infinitesimal generators presenting several examples, which illustrate  various possibilities, and introduce regular fractional singularities. Finally, in Section~\ref{S geometric} we give our geometric criteria on the image of the K\oe nigs function in order to have regular fractional singularity and present a characterization of regular poles in terms of Bertilsson's $\beta$-point sufficient criterion.

\newpage\section{Preliminaries}\label{S prel}
\subsection{Contact, fixed and $\beta$- points}\label{SS_boundaryreg}
For the unproven statements, we refer the reader to, {\sl e.g.}, \cite{Abate}, \cite{CMbook} or \cite{Shb}.

Let $f:\D \to \D$ be holomorphic, $x\in \de \D$, and let
\[
\al_x(f):=\liminf_{z\to x} \frac{1-|f(z)|}{1-|z|}.
\]
From Julia's lemma it follows that $\al_x(f)>0$. The number $\al_x(f)$ is called the {\sl boundary dilatation coefficient} of $f$.

If $f:\D \to \C$ is a map and $x\in \de\D$, we write $\angle\lim_{z\to x}f(z)$ for the non-tangential (or angular) limit of $f$ at $x$. If this limit exists and it cannot cause any confusion, we will suppress the language and denote its value simply by~$f(x)$.

\begin{definition}
Let $f:\D \to \D$ be holomorphic. A point $x\in \de \D$ is said to be a \textsl{contact point} (respectively, \textsl{boundary fixed point}) of~$f$, if $f(x):=\angle \lim_{z\to x}f(z)$ exists and belongs to~$\UC$ (respectively, coincides with~$x$). If in addition
\begin{equation}\label{EQ_reg-contact-point}
\al_f(x)<+\infty,
\end{equation}
then $x$ is called a \textsl{regular contact point} (respectively, \textsl{boundary regular fixed point}) of $f$.

A boundary fixed point which is not regular is said to be \textsl{super-repulsive}.
\end{definition}
\begin{remark}
By the Julia\,--\,Wolff\,--\,Carath\'eodory theorem, condition~\eqref{EQ_reg-contact-point} in the above definition is sufficient on its own for~$x\in\UC$ to be a regular contact point of~$f$.
\end{remark}

If $f:\D\to\D$ is holomorphic, neither the identity nor an elliptic automorphism, by the Denjoy\,--\,Wolff theorem, there exists a unique point $\tau\in\oD$, called the {\sl Denjoy\,--\,Wolff point} of $f$ (or abbreviated, the \textsl{DW-point}), such that $f(\tau)=\tau$ and the sequence of iterates $\{f^{\circ k}\}$ converges uniformly on compacta of $\D$ to the constant map $z\mapsto \tau$.

\begin{definition}
Let $f:\D \to \C$ be holomorphic. A point $x\in \de \D$ is said to be a {\sl $\beta$-point} of~$f$ if
\[
\angle \limsup_{z\to x} \frac{|f'(z)|}{|x-z|}=L<+\infty.
\]
\end{definition}

If $x\in \de \D$ is a $\beta$-point of~$f$, then $f$ has non-tangential limit $f(x)\in \Complex$ and $f'$ has non-tangential limit $f'(x)=0$ at~$x$ (see, {\sl e.g.} \cite[Section~3.2]{BCD-M_regular-poles}). If in addition $f(\UD)\subset\UD$, then $f(x)\in\UD$.

\subsection{One-parameter semigroups and infinitesimal
generators}\label{SS_semig}

A one-parameter semigroup $(\phi_t)$ of holomorphic self-maps of~$\UD$ is a continuous homomorphism $t\mapsto \phi_t$ from the
additive semigroup $(\R_{\ge0}, +)$ of non-negative real numbers to the
semigroup $({\sf Hol}(\D,\D),\circ)$ of all holomorphic self-maps
of $\D$ with respect to the composition, endowed with the
topology of uniform convergence on compacta.

By Berkson and Porta's theorem \cite[Theorem~(1.1)]{Berkson-Porta}, if $(\phi_t)$
is a one-parameter semigroup in~$\Hol(\UD,\UD)$, then $t\mapsto \phi_t(z)$
is real-analytic and there exists a unique holomorphic vector field
$G:\D\to \C$ such that
\[
\frac{\de \phi_t(z)}{\de
t}=G(\phi_t(z))\qquad\text{for all $z\in\UD$ and all~$t\ge0$}.
\]
This vector field $G$, called the  {\sl infinitesimal generator} of $(\phi_t)$,  is {\sl semicomplete} in the sense that the Cauchy problem
\[
\begin{cases}
\mydot x(t)=G(x(t)),\\
x(0)=z,
\end{cases}
\]
has a  solution $x^z:[0,+\infty)\to \D$ for every $z\in \D$.
Conversely, any semicomplete holomorphic vector field in $\D$
generates a one-parameter semigroup in ${\sf Hol}(\D,\D)$.

We denote by $\Gen$ the set of all infinitesimal generators in $\D$.

\begin{remark}
It is known that if $(\phi_t)$ is a one-parameter semigroup  and ${\phi_{t_0}\in\Aut}$ for some~$t_0>0$, then $(\phi_t)\subset\Aut$ and it can be extended to a one-parameter group in~$\Aut$, which is a very well studied and understood object.
\end{remark}
Therefore, in what follows we adopt the following assumption.\\[2mm]
\textbf{Assumption.} For all one-parameter semigroups~$(\phi_t)$  in~$\Hol(\UD,\UD)$ considered throughout this paper, we suppose that $\phi_t\not\in\Aut$ for all~$t>0$.\vskip2mm

Let $G$ be an infinitesimal generator with
the associated one-parameter semigroup~$(\phi_t)$. Then there exists a unique
$\tau\in\oD$ and a unique holomorphic function $p:\D\to \C$ with $\Re
p(z)\geq 0$ such that the following identity, known as the {\sl Berkson\,--\,Porta
formula}, takes place
\begin{equation}\label{EQ_B-P-representation}
G(z)=(z-\tau)(\overline{\tau}z-1)p(z)\quad\text{for all $z\in\UD$}.
\end{equation}
The point $\tau$ in the Berkson\,--\,Porta formula turns out to be
the  Denjoy\,--\,Wolff point of $\phi_t$ for all $t>0$.
Moreover, if $\tau\in \de\D$, then  $\angle\lim_{z\to
\tau}\phi_t'(z)=e^{\beta t}$ for some $\beta\leq 0$.

\begin{definition}
A {\sl boundary regular fixed point}  for a one-parameter semigroup
$(\phi_t)$ is a point $x\in \de \D$ which is a boundary regular fixed point of~$\phi_t$ for all~$t>0$.
\end{definition}
\begin{definition}
A \textsl{boundary regular null point} for an infinitesimal generator $G$, is a point $x\in \de \D$ such that
\[
\angle\lim_{z\to x}\frac{G(z)}{z-x}=\ell\in \R,
\]
exists finitely. The number $\ell$ is called the {\sl dilation} of $G$ at $x$.
\end{definition}

It is well known (see, \cite[Theorem 1]{CDP}, \cite[Theorem 2]{CDP2}, \cite[pag. 255]{Siskakis-tesis}, \cite{ES}) that given a one-parameter semigroup $(\phi_t)$ with associated infinitesimal generator $G$, a point $x\in \de \D$ is a boundary (regular) fixed point of $\phi_{t_0}$ for some $t_0>0$ if and only if it is a boundary (regular) fixed point of $\phi_t$ for all $t\geq 0$. Moreover, $x$ is a boundary regular fixed point for $(\phi_{t})$ if and only if $x$ is a boundary regular null point for $G$. Boundary regular fixed points can be characterized using the Poisson kernel (also in higher dimensions, see \cite{BCD-M_Pluripotential}) or via Berkson\,--\,Porta type formulas \cite{BCD3, Goryainov-Kudryavtseva, Shoikhet2}.

Now we recall what a regular pole is and its characterization in terms of $\beta$-points.

\begin{definition}
Let $G\in \Gen$.  A point $x\in \de \D$ is a {\sl  regular pole of $G$ of  mass $C>0$} if
\[
\angle\liminf_{z\to x} |G(z)(x-z)|= C.
\]
\end{definition}

As shown in \cite{BCD-M_regular-poles}, if $x$ is a regular pole for $G\in \Gen$ then $\angle\lim_{z\to x} G(z)(x-z)$ exists finite and nonzero. In \cite{BCD-M_regular-poles} the relations between regular poles and $\beta$-points were studied and it was proved that given a one-parameter semigroup $(\phi_t)$ with associated infinitesimal generator $G$, a point $x\in \de \D$ is a regular pole of $G$ if and only if $x$ is a $\beta$-point for $\phi_t$ for some~--- hence for all~--- $t>0$.

To any one-parameter semigroup in~$\Hol(\UD,\UD)$ one can associate a (unique) intertwining map which simultaneously linearizes all the maps in the semigroup. The ideas for the proof of the following result are in \cite{H}, and, with different methods in \cite{Berkson-Porta} and \cite{Siskakis-tesis} (see also \cite[Chapter 1.4]{Abate}).

\begin{proposition}
\label{PR_Koenigs-function} Let $(\phi_{t})$ be a non-trivial
semigroup in $\mathbb{D}$ with infinitesimal generator~$G$.
Then there exists a unique univalent function $h:\D\to\C$,
called the {\sl K\oe nigs function} of~$(\phi_t)$, such that
\begin{mylist}
\item[(A)] If $(\phi_{t})$ has the Denjoy\,--\,Wolff point $\tau \in \mathbb{D}$ then
$h(\tau)=0$, $h'(\tau)=1$ and
\\${h(\phi_t(z))=\exp(G'(\tau)t)h(z)}$ for all $t\geq 0$ and $z\in\UD$.
Moreover, $h$ is the unique holomorphic function from
$\mathbb{D}$ into $\mathbb{C}$ such that
\begin{enumerate}
\item[(i)] $h^{\prime }(z)\neq 0$ for every $z\in \mathbb{D},$
\item[(ii)] $h(\tau )=0$ and $h^{\prime }(\tau )=1,$
\item[(iii)] $h^{\prime }(z)G(z)=G^{\prime }(\tau )h(z),$ for
every $z\in \mathbb{D}.$
\end{enumerate}
\item[(B)] If $(\phi_{t})$ has the Denjoy\,--\,Wolff point $\tau
\in \partial \mathbb{D}$ then $h(0)=0$ and
$h(\phi_t(z))=h(z)+t$ for all $t\geq 0$  and $z\in\UD$. Moreover, $h$ is
the unique holomorphic function from $\mathbb{D}$ into~$\mathbb{C}$ such  that:
\begin{enumerate}
\item[(i)] $h(0)=0,$
\item[(ii)] $h^{\prime }(z)G(z)=1$ for every $z\in \mathbb{D}.$
\end{enumerate}
\end{mylist}
\end{proposition}

The following result is a form of the Lindel\"of theorem we will need later.
\begin{theorem}[Lindel\"of]\label{TH_Lind}
Let $f:\UD\to\Complex$ be a holomorphic function and let $\sigma\in\UC$. Suppose that one of the following assertions hold:
\begin{mylist}
\item[(i)] $\Complex\setminus f(\UD)$ contains at least two different points,
\item[(ii)] $f$ is univalent in~$\UD$,
\item[(iii)] $f$ is the derivative of a univalent function in~$\UD$, or
\item[(iv)] $f\in\Gen$.
\end{mylist}
If there exists a continuous map $\gamma:[0,1)\to\UD$ with $\lim_{[0,1)\ni s\to 1}\gamma(s)=\sigma$  and such that the limit $a:=\lim_{s\to1}f(\gamma(s))$ exists, finite or infinite, then the angular limit~$\anglim_{z\to\sigma}f(z)$ also exists and equals~$a$.
\end{theorem}
More generally, the Lindel\"of theorem holds for all normal meromorphic functions, see, \textit{e.g.}, \cite[\S9.1, Theorem\,9.3]{Pommerenke}. The fact that functions~$f$ satisfying (i), (ii), or~(iii) are normal follows from~\cite[Lemmas~9.2 and~9.3  in \S9.1]{Pommerenke}, and to show that all $f\in\Gen$ are normal one can use the argument in the proof of Lemma~\ref{LM_ang_lim_at_frac_sing}, which we give later.

In the sequel we need also this kind of ``boundary Lindel\"of theorem'':
\begin{lemma}\label{LM_Lind}
Suppose that $f\in\Hol(\UD,\Complex)$ extends holomorphically to an open arc~$A\subset\partial\UD$.
Let $\sigma$ be one of the end-points of~$A$.
Then the following statements hold.
\begin{mylist}
\item[(A)] If $f$ satisfies one of conditions (i)\,--\,(iv) in Theorem~\ref{TH_Lind} and the limit $\lim_{A\ni x\to\sigma} f(x)$ exists, finite or infinite, then the angular limit $\anglim_{z\to\sigma}f(z)$ also exists and both limits coincide.

\item[(B)] If $f\in \Gen$ and $\Re\{ \overline{x} f(x)\}=0$ for all $x\in A$, then the limit $\lim_{A\ni x\to\sigma} f(x)$ exists, finite or infinite.

\item[(C)] If $f$ is univalent in~$\UD$ and the set $f(A)$ is contained on a Jordan curve $\Gamma\subset\ComplexE$, the limit $\anglim_{z\to\sigma}f(z)$ exists, finite or infinite.

\end{mylist}
\end{lemma}
\begin{proof}
First of all, assertion (A) follows easily from Lindel\"of Theorem~\ref{TH_Lind}. To prove (B) write Berkson\,--\,Porta formula~\eqref{EQ_B-P-representation}: $f(z)=(\tau-z)(1-\overline\tau z)p(z)$, where $\tau\in\overline\UD$ and $p$ is holomorphic in~$\UD$, with $\Re p\ge0$.  We may assume that $\tau\not\in A$. Otherwise, instead of~$A$, we would consider the subarc of $A$ between $\sigma$ and~$\tau$. Then, taking into account that $\Re\{\overline z f(z)\}=0$ for all $z\in A$, we see that $p$ extends holomorphically to $D:=\ComplexE\setminus(\partial\UD\setminus A)$, with $p(1/\overline z)=-\overline{p(z)}$ for all $z\in\UD$ and $\Re p|_A=0$. Clearly, we may assume that $p$ is not constant. Now, if $\Phi$ is the appropriately chosen conformal map of $\UD$ onto $D$ that takes $(-1,1)$ to $A$, then $g:=ip\,\circ \Phi$ is a typically real function and hence, see, \textit{e.g.}, \cite[p.\,55--56]{Duren}, $g(z)=g(0)+zp_1(z)/(1-z^2)$ for all $z\in\UD$, where $p_1$ is a holomorphic function in $\UD$ with $\Re p_1\ge0$. Bearing in mind that $1/p_1$ is also a holomorphic function with non-negative real part unless $p_1\equiv0$, we conclude with the help of \cite[Lemma~3.2]{BCD-M_regular-poles}, see also \cite[Ch.\,IV~\S26]{Valiron}, that $g|_{(-1,1)}$ has limits at the points~$\pm1$, finite or infinite. This proves~(B).

It remains to prove (C). Fix any two points $\sigma_1\neq\sigma_2$ on~$A$. Denote by $S(\sigma_1,\sigma_2)$  the open disk sector  bounded by the segments $[0,\sigma_1]$, $[0,\sigma_2]$ and the closed subarc~$A(\sigma_1,\sigma_2)$ of $A$ between~$\sigma_1$ and~$\sigma_2$. The set $\partial f(S(\sigma_1,\sigma_2))$ is the union of the Jordan arc $f\big([0,\sigma_1]\cup[0,\sigma_2]\big)$ and a closed connected subset $f\big(A(\sigma_1,\sigma_2)\big)\supset\{f(\sigma_1),\, f(\sigma_2)\}$ of the Jordan curve~$\Gamma$. Therefore, according to \cite[Prop.\,2.5 on p.\,23]{Pommerenke2} each point of~$\partial f(S(\sigma_1,\sigma_2))$ has at most two preimages w.r.t. the restriction of~$f$ to the closure of~$S(\sigma_1,\sigma_2)$. Since $\sigma_1$ and $\sigma_2$ are chosen arbitrarily, it follows that each value of $f|_A$ is taken at most twice. Since $f|_A\colon A\to\Gamma$ is continuous, this implies that $f|_A$ has limits, finite or infinite, at both end-points of~$A$. The proof is now complete.
\end{proof}

Now we collect some results from \cite{CDP,Pavel} which we need in the following.

\begin{proposition}\label{PR_Pavel}
Let $(\phi_{t})$ be a non-trivial
one-parameter semigroup in $\mathbb{D}$ and ${h:\D\to\C}$ the K\oe nigs function of~$(\phi_t)$. The following statements hold:
\begin{mylist}
  \item[(i)] For any $t\geq 0$ and $\sigma\in \de \D$ there exists the angular limit $\phi_t(\sigma):=\angle\lim_{z\to \sigma}\phi_t(z)$.
  \item[(ii)] For any $\sigma\in\partial\UD$ the map $[0,+\infty)\ni t\mapsto \angle\lim_{z\to \sigma}\phi_t(z)$ is continuous.
  \item[(iii)] For any $\sigma\in \de \D$ there exists the angular limit $h(\sigma):=\angle \lim_{z\to \sigma} h(z)\in \C\cup\{\infty\}$. Moreover, $\sigma\in \de \D$ is a boundary fixed point of~$(\phi_t)$ if and only if $h(\sigma)=\infty$.
\end{mylist}
\end{proposition}
The proof of~(i) can be found in \cite[Proof of Theorem\,5, p.\,479]{CDP} and~\cite[Theorem\,3.1]{Pavel}. Assertion (ii) is contained in~\cite[Prop.\,3.2]{Pavel}. Existence of angular limits of $h$ in assertion~(iii) is a consequence of the special geometric properties of K\oe nigs functions, see, \textit{e.g.}, \cite[Prop.~3.4 and comments before]{Pavel}. Finally, the part of~(iii) regarding the boundary fixed points is by \cite[Theorem\,2]{CDP} and \cite[Remark\,3.5]{Pavel}.
\begin{remark}\label{RM_Pavel}
We make a small but useful addition to statement~(i) of the above proposition:
the Lindel\"of theorem applied to $\phi_s$ along the curve~$[0,1)\ni r\mapsto \phi_t(r\sigma)$ shows that the identity $\phi_{s}(\phi_t(\sigma))=\phi_{t+s}(\sigma)$, where $\phi_s(\phi_t(\sigma))$ is to be understood as the angular limit of~$\phi_s$ at the point~$\phi_t(\sigma)$, holds also for all boundary points~$\sigma\in\UC$.
\end{remark}
\begin{remark}\label{RM_no-extra-preim-for-FPs}
One more interesting statement is that if $\phi_{t_0}(\sigma)$ is a boundary fixed point of~$(\phi_t)$ for some $\sigma\in\UC$ and $t_0\ge0$, then $\sigma=\phi_{t_0}(\sigma)$ and thus $\sigma$ is a boundary fixed point itself. Otherwise, according to the previous remark and statement (ii) of the above proposition there would exist a curve, namely the arc $L:=\{\phi_t(\sigma):\,t\in[0,t_0]\}$, joining $\sigma$ with $\phi_{t_0}(\sigma)\neq\sigma$ such that $\phi_{t_0}(L)=\{\phi_{t_0}(\sigma)\}$. This would imply that $\phi_{t_0}$ were constant by Fatou's theorem.
\end{remark}

\section{Contact points which are not fixed}\label{S contact}
In this section we study the relation between non-fixed contact points of one-parameter semigroups and the boundary behaviour of the associated infinitesimal generators. We start with the following definition.

\begin{definition}\label{DF_contact-arc}
An open arc $A\subset\UC$ is said to be a \textsl{contact arc} for an infinitesimal generator~$G$, if
$G$ extends holomorphically to $A$ with $\Re\{\overline z\,G(z)\}=0$ and $G(z)\neq0$ for all $z\in A$.
Each contact arc~$A$ is endowed with a natural orientation induced by the unit vector field $G/|G|$. This determines  the \textsl{initial} and \textsl{final} end-points of~$A$.
\end{definition}

Now we prove our main result about contact points:

\begin{proof}[Proof of Theorem \ref{TH_contact-points}]
Note, first of all, that using \cite[Proposition 2.1]{Pavel} we may assume that the DW-point of $(\phi_t)$ is $\tau=1$.

\step1{(i) implies (ii)} Let $t_1:=\sup\{{t \geq 0}: \phi_t(x)\in \de \D\}$. If $\phi_s(x)\in \D$ for some $s\in (0,+\infty)$, then obviously $\phi_t(x)\in \D$ for all $t\geq s$. Therefore $t_1\in [t_0,+\infty]$ and $\phi_t(x)\in \de \D$ for all $t\in [0,t_1)$. Let $A:=\{\phi_t(x): t\in (0,t_1)\}\subset \de \D$. Note that by Remark~\ref{RM_no-extra-preim-for-FPs}, $A$ does not contain boundary fixed points of $(\phi_t)$.

As a consequence, the map $\gamma:[0,t_1)\ni t\mapsto \phi_t(x)$ is injective. Otherwise we would have $\phi_t(x)=\phi_{t+s}(x)$ for some $t\in[0,t_1)$ and $s\in(0,t_1-t)$, which by Remark~\ref{RM_Pavel} means that $\phi_t(x)$ is a boundary fixed point of~$(\phi_t)$. The map~$\gamma$ is also continuous by Proposition~\ref{PR_Pavel}(ii). Therefore, $A$ is an open arc and $x$ is one of its end-points.

Now fix $s\in (0,t_1)$ and let $A_s:=\{\phi_t(x): t\in (0,s)\}$. Denote by $S$ the open sector $\{r y: y\in A_s, r\in (0,1)\}$ and let $U:=h(S)$, where $h$ is the K\oe nigs function associated to~$(\phi_t)$. By Proposition~\ref{PR_Pavel}(iii), $h$ has a finite non-tangential limit at $x$. With the help of the Abel equation  $h(rx)+t=h(\phi_t(rx))$ for $r\in [0,1)$, see Proposition~\ref{PR_Koenigs-function}(B), it follows that $\lim_{r\to 1} h(\phi_t(rx))=h(x)+t$. Then, by  Lindel\"of Theorem~\ref{TH_Lind} applied to $h$ along the curve~$[0,1)\ni r\mapsto \phi_t(rx)$, we have
$h(\phi_t(x)):=\anglim_{z\to\phi_t(x)}h(z)=h(x)+t.$ Therefore the set $\{h(\sigma)\colon \sigma\in A_s\}$ is the interval $L_0:=\big(h(x),h(x)+s\big)\subset \C$. We claim that $\de U=L_0\cup L_1\cup L_2$, where $L_1:=\{h(rx): r\in [0,1]\}$ and $L_2:=\{h(r\phi_{s}(x)): r\in [0,1]\}$.  Indeed, it is known (see, \textit{e.g.}, \cite[p.\,35\,--\,39]{Goluzin}) that if $f:\D \to \C$ is univalent, then the set
\[
\mathcal C:=\Big\{c\in \C\cup\{\infty\}: \exists\, \sigma\in \de \D: \lim_{(0,1)\ni r\to 1}f(r\sigma)=c\Big\},
\]
is dense in $\de f(\D)$. This fact applied to the composition of $h|_S$ with a conformal mapping of $\D$ onto $S$, implies our claim, because the set $L_0\cup L_1\cup L_2$ is closed.

Therefore, $U$ is a Jordan domain and by the Carath\'eodory extension theorem, see, \textit{e.g.}, \cite[Thm.\,2.6 on p.\,24]{Pommerenke2}, $h:S\to U$ extends to a homeomorphism of~$\overline{S}$ onto $\overline{U}$, which maps $A_s$ onto~$L_0$. By the Schwarz reflection principle, $h$ extends holomorphically through~$A_s$. By the arbitrariness of $s\in(0,t_1)$, in fact $h$ extends holomorphically through the whole arc $A$. Since $h'(z)G(z)=1$ for all $z\in\UD$ (see Proposition~\ref{PR_Koenigs-function}(B)), also $G$ extends holomorphically through $A$. Moreover, bearing in mind that $h(A)$ is an interval parallel to the real axis, with $h(x)$ being its left end-point, and using again $h'(z)G(z)=1$, we see that $A$ is, in fact, a contact arc for~$G$ with  initial point $x$.

\step2{(ii) implies (i) and (iii)}
Consider the local flow of $G$ on the arc~$A$. Since $G$ is holomorphic in a neighborhood of~$A$ and tangent to~$A$ at every point, it follows that for each $y\in A$, there exists the maximal solution $w_y:I_y\to A$  to the initial value problem
$$
\mydot w_y\!(t)=G|_A(w_y(t)),\quad t\ge0,\qquad w_y(0)=y,
$$
and that $\phi_t$ has a holomorphic extension to a neighborhood of~$y$ with $\phi_t(y)=w_y(t)$ for all $t\in I_y$.
Fix any $y_0\in A$ and let $A_1$ be the subarc of $A$ between $x$ and $y_0$. It is easy to see that $I_{y_0}\subset I_y$ for any $y\in A_1$. Therefore, for all $t\in I_{y_0}$ the function $\phi_t$ extends homeomorphically to $A_1$, with $\phi_t(A_1)\subset\partial\UD$. Recall that all $\phi_t$'s are univalent in~$\UD$. Therefore, according to Lemma~\ref{LM_Lind} there exists $\lim_{A_1\ni y\to x}\phi_t(y)=\phi_t(x)$. Since $\phi_t(x)\neq x$ for all $t>0$, it follows that $\phi_t(x)\in A$ for all $t\in I_{y_0}\setminus\{0\}$. This proves~(i).

Now fix any $s\in I_{y_0}$ and let $y_1:=\phi_s(x)$, $t_1:=s+\sup I_{y_1}$. Since $G$ has no zeros on~$A$, it follows that $(s,t_1)\ni t\mapsto \phi_t(x)=w_{y_1}(t-s)$ is a homeomorphic map onto the open subarc of $A$ between $y_1$ and the final point of~$A$. This proves (iii), because by Proposition~\ref{PR_Pavel}(ii), $\phi_s(x)\to x$ as $s\to0^+$.
\end{proof}

\begin{definition}
Let $G: \D \to \C$ be an infinitesimal generator. A contact arc $A\subset \de \D$ for $G$ is called {\sl maximal} if there exists no other contact arc $A_1$ for $G$ such that $A_1\supsetneq A$.
\end{definition}

By Theorem \ref{TH_contact-points} to any contact not fixed boundary point of $(\phi_t)$ there corresponds a unique maximal contact arc for $G$.

We classify maximal contact arcs and describe more in details the dynamics of the semigroup on them:

\begin{proposition}\label{PR_contact_arcs}
Let $(\phi_t)$ be a one-parameter semigroup with  Denjoy\,--\,Wolff point ${\tau\in\oD}$ and  associated infinitesimal generator~$G$. Suppose that $A_0\subset\UC$ is a maximal contact arc for~$G$. Let $x_0$ and $x_1$ be the initial and final points of~$A_0$, respectively. Then the following statements hold:
\begin{mylist}
\item[(i)] one of the two alternatives takes place:
\begin{itemize}
\item[(i.1)] $x_0$ is a boundary fixed point of~$(\phi_{t})$, or
\item[(i.2)]  $x_0$ is a contact (not fixed) point of~$\phi_{t_0}$ for some $t_0>0$;
\end{itemize}

\item[(ii)] one of the two alternatives takes place:
\begin{itemize}
\item[(ii.1)] $x_1=\tau$, or
\item[(ii.2)] $\phi_t(x_1)\in \D$ for all $t>0$;
\end{itemize}
\item[(iii)] if $x\in A_0$, then the life-time $t_1(x,A_0)<+\infty$  if and only if $\phi_t(x_1)\in \D$ for all $t>0$;
\end{mylist}
Moreover,
\begin{mylist}
\item[(a1)]  the limit $G(x_0):=\lim_{A_0\ni x\to x_0}G(x)$ exists  finitely;
\item[(a2)] if $x_0$ is a boundary fixed point of~$(\phi_{t})$, then $G(x_0)=0$;
\item[(b1)] the limit $G(x_1):=\lim_{A_0\ni x\to x_1}G(x)$ exists, finite or infinite;
\item[(b2)] $G(x_1)=0$ if and only if $x_1=\tau$.
\end{mylist}
\end{proposition}
\begin{remark}
It is evident from the proof given below that assertion~(iii) of Proposition~\ref{PR_contact_arcs}
is also true for $x:=x_0$ provided $x_0$ is not a boundary fixed point of~$(\phi_t)$.
\end{remark}

\begin{proof}[Proof of Proposition~\ref{PR_contact_arcs}.]
Statement~(i) follows readily from Theorem~\ref{TH_contact-points}. To prove  statements (ii) and (iii), let first assume $t_1:=t_1(x,A_0)<+\infty$ for some $x\in A_0$. By Theorem~\ref{TH_contact-points}(iii), $\lim_{t\to t_1^-}\phi_t(x)=x_1$. Then by Proposition~\ref{PR_Pavel}(ii), $\phi_{t_1}(x)=x_1$. Consequently, $\phi_t(x_1)\in\UD$ for all $t>0$. Indeed,  if $x_1$ were a contact point of $\phi_{s}$ for some $s>0$, then  according to Remark~\ref{RM_Pavel}, $x$ would be a contact point of~$\phi_{t_1+t}$ for all $t\in [0,s]$. Together with the fact that $x$ is not a boundary fixed point of~$(\phi_t)$ this would contradict  the maximality of $A_0$.

Next, assume that $x\in A_0$ and $t_1:=t_1(x,A_0)=+\infty$. Then $x$ is a contact point of~$\phi_t$ for all~$t\ge0$, but not a boundary fixed point of~$(\phi_t)$. Therefore, by~\cite[Theorem\,\,4]{CDP}, $\tau\in\UC$ and hence, by~\cite[Remark~5.1]{Pavel}, $\phi_t(x)\to\tau$ as $t\to+\infty$. Since  $\lim_{t\to t_1^-}\phi_t(x)=x_1$, it follows that $x_1=\tau$.

In order to prove (a1), (a2), (b1) and (b2), using \cite[Proposition 2.1]{Pavel}, we will assume that the Denjoy\,--\,Wolff point of $(\phi_t)$ is $\tau=1$. Then, according to Proposition~\ref{PR_Koenigs-function}(B).(ii), the K\oe nigs function~$h$ of~$(\phi_t)$ extends holomorphically to~$A_0$, with $h(A_0)$ being a subset of a straight line.

Applying Lemma~\ref{LM_Lind} to $f:=G$ and $A:=A_0$ one proves~(b1) and a part of statement~(a1).
In order to complete the proof of (a1) and prove (a2), we suppose first that $x_0$ is a boundary fixed point of~$(\phi_t)$. Then by Proposition~\ref{PR_Pavel}(iii) and Lemma~\ref{LM_Lind} applied to $f:=h$, we have $h(x)\to\infty$ as $A_0\ni x\to x_0$. It follows that $\limsup_{A_0\ni x\to x_0}|h'(x)|=+\infty$, which in view of Proposition~\ref{PR_Koenigs-function}(B).(ii) implies that $\lim_{A_0\ni x\to x_0} G(x)=0$, proving (a2) and (a1) in this case.

Now suppose that $x_0$ is not a fixed point. By the argument in Step~1 in the proof of Theorem~\ref{TH_contact-points}, $\partial h(\UD)$ contains a segment $[h(x_0),h(x_0)+s]$ for some~$s>0$. It follows, in particular, that the ray $R:=\{h(x_0)+s-\xi\colon \xi\ge0\}$ does not intersect~$h(\UD)$ and hence $\varphi:=h_0^{-1}\circ h$, where $h_0$ is the conformal map of $\UD$ onto $\Complex\setminus R$, is a well-defined self-map of $\UD$ and has a contact point at~$x_0$. By the Julia\,--\,Wolff\,--\,Carath\'eodory theorem the angular derivative $\varphi'(x_0)$ exists, finite or infinite, and does not vanish. This implies that the angular derivative of $h$ at~$x_0$ exists and does not vanish.  According to Proposition~\ref{PR_Koenigs-function}(B).(ii) and Lemma~\ref{LM_Lind}, it follows that $G(x_0)\neq\infty$, which completes the proof of~(a1).

In order to prove (b2), we first assume $x_1=\tau=1$. In this case by \cite[Theorem 1]{CDP2}, $\anglim_{z\to\sigma}G(z)=0$ and hence by Lemma~\ref{LM_Lind} applied to~$f:=G$, the limit of $G|_A$ at~$x_1$ also (exists and) equals~$0$.

In case $x_1$ is not the Denjoy\,--\,Wolff point of $(\phi_t)$, by Proposition~\ref{PR_Pavel}(iii), $h(x_1)\in\Complex$. Take any $x\in A_0$. According to what we already proved and by the argument in Step~1 in the proof of Theorem~\ref{TH_contact-points} it is easy to see that:
\begin{mylist}
\item[(a)] there exists $t_1:=t_1(x,A_0)>0$ such that $\phi_{t_1}(x)=x_1$;
\item[(b)] $h(x_1)=h(x)+t_1$;
\end{mylist}
Moreover,
\begin{mylist}
\item[(c)] if $w\in h(\UD)$, then $w+\xi=h\big(\phi_{\xi}(h^{-1}(w))\big)\in h(\UD)$ for all $\xi\ge0$;
\end{mylist}
and in particular,
\begin{mylist}
\item[(d)] the curve $[0,1)\ni r \mapsto h(rx)$ does not intersect the ray $\{h(x_1)-\xi:\,\xi\ge0\}$.
\end{mylist}

By (a) and (b) there exists $r_1\in(0,1)$ such that $u_0:=\max_{r_1\le r<1} \Re h(rx)<\Re h(x)+t_1=\Re h(x_1)$. Let $r_0$ be any of the points at which $[r_1,1)\ni r\mapsto |\Im h(rx)-\Im h(x)|$ achieves its maximum value, and let $v_0:=\Im h(r_0x)$. Note that $v_0\neq\Im h(x)$. Choose now $r_2\in(r_1,1)$ such that $t_1+\Re h(rx)+t_1>u_0$ for all $r\in[r_2,1)$. Taking into account (b)\,--\,(d), we conclude that:
\begin{mylist}
\item[(e)] the semistrip $\Pi$ bounded by the rays $$\{w:\,\Im w=\Im h(x),\,\Re w\ge u_0\},\quad \{w:\,\Im w=v_0,\,\Re w\ge u_0\}$$ and the segment $[u_0+iv_0,u_0+i\Im h(x)]$, is contained in $h(\UD)$;
\item[(f)] the curve $\gamma:[r_2,1)\ni r\mapsto h(rx)+t_1$ is a slit in~$\Pi$.
\end{mylist}
From the Abel equation $h(rx)+t_1=h(\phi_{t_1}(rx))$ for all $r\in[r_2,1)$ it follows that
\begin{mylist}
\item[(g)] $h^{-1}\circ\gamma$ is a slit in~$\UD$ landing at~$\phi_{t_1}(x)=x_1$.
\end{mylist}
Now by \cite[Thm.\,10.6 on p.\,307]{Pommerenke}, $h$ has a finite angular derivative at~$x_1$. Therefore, thanks to equality~(ii) in Proposition~\ref{PR_Koenigs-function}(B), $\anglim_{z\to x_1} G(z)$ exists, finite or infinite, but cannot be equal to~$0$. Applying Lemma~\ref{LM_Lind} to $f:=G$, we conclude that the same holds for $\lim_{A\ni \sigma\to x_1} G(\sigma)$. The proof is now complete.
\end{proof}

By Theorem \ref{TH_contact-points} and Proposition \ref{PR_contact_arcs}, for every $t>0$ the map $\phi_t$ can have at most a countable number of contact non-fixed points in which $\phi_t$ has no  holomorphic extension (the initial points of maximal contact arcs), hence:

\begin{corollary}\label{CR_non-embedd}
If $\varphi\in\Hol(\UD,\UD)$ has an uncountable set of contact points which are not boundary fixed points of~$\varphi$ and at which~$\varphi$ has no holomorphic extension, then $\varphi$ cannot be embedded in a one-parameter semigroup.
\end{corollary}

An example of  a \textsl{univalent map} $\varphi$ continuous up to~$\de\UD$ and having an uncountable set of non-fixed contact points at which~$\varphi$ has no holomorphic extension, can be obtained by considering the one-parameter semigroup $(\phi_t)$ with uncountable set of boundary super-repulsive fixed points defined in~\cite[p.\,260]{AnalyticFlows} and setting $\varphi:=i\phi_t$ for any arbitrary $t>0$.

We conclude this section showing that the holomorphic extendability of the infinitesimal generator in the definition of a contact arc can be replaced by a weaker condition formulated in terms of radial limits.

\begin{proposition}\label{PE_rad-lims-generator}
Let $G$ be an infinitesimal generator in~$\UD$ and let $A\subset\UC$ be an open arc.  Then $A$ is a contact arc for $G$ if and only if $\lim_{(0,1)\ni r\to 1}\Re\{\overline\sigma\, G(r\sigma)\}=0$  and $\limsup_{(0,1)\ni r\to 1}|\Im G(r\sigma)|\neq 0$ for any $\sigma\in A$.
\end{proposition}

This follows at once from the following lemma, which might be interesting by its own:

\begin{lemma}\label{LE_rad-lims-generator}
Let $G$ be an infinitesimal generator in~$\UD$ and let $A\subset\UC$ be an open arc.  Suppose that
$\lim_{(0,1)\ni r\to 1}\Re\{\overline\sigma\, G(r\sigma)\}=0$ for any $\sigma\in A$, then $G$ has a holomorphic extension to~$A$.
\end{lemma}

\begin{proof}
We first prove the lemma for the case when the Denjoy\,--\,Wolff point $\tau\not\in A$.

Fix any $\sigma\in A$.
First of all notice that
\begin{equation}\label{EQ_no-reg-poles}
\ell(\sigma):=\lim_{(0,1)\ni r\to 1}G(r\sigma)(1-r)=0.
\end{equation}
Indeed, by \cite[Lemma~3.3]{BCD-M_regular-poles} the limit in~\eqref{EQ_no-reg-poles} exists finitely and $\overline\sigma\ell(\sigma)\in\Real$, which contradicts the hypothesis unless $\ell(\sigma)=0$.

Write the Berkson\,--\,Porta formula $G(z)=q(z)p(z)$, where $q(z):=(\tau-z)(1-\overline\tau z)$ and $p$ is a Herglotz function. Note that $\overline{\sigma}\,q(\sigma)=-|\tau-\sigma|^2\in\Real$. Therefore, according to~\eqref{EQ_no-reg-poles}, $$\Big|\Re\{\overline{\sigma}\,G(r\sigma)\}+|\tau-\sigma|^2\Re p(r\sigma)\Big|\le\left|1+\frac{|\tau-\sigma|^2}{\overline{\sigma}\,q(r\sigma)}\right|\cdot|G(r\sigma)|\to 0$$ as $(0,1)\ni r\to 1$. It follows that the hypothesis of the lemma is equivalent to
\begin{equation}\label{EQ_cont-for-p}
\lim_{(0,1)\ni r\to1}\Re p(r\sigma)=0\quad \text{for all $\sigma\in A$}.
\end{equation}
Now write the Herglotz representation of~$\Re p$,
\begin{equation}\label{EQ_Herglotz}
\Re p(re^{i\theta})=\int_{[-\pi,\pi)}\mathcal P_r(\theta-t)\,d\nu(t) \quad \text{for all $r\in[0,1)$ and $\theta\in\Real$},
\end{equation}
where $\nu$ is a positive finite Borel measure on~$[-\pi,\pi)$ and $\mathcal P_r$ is the Poisson kernel.

By a technical reason, without loss of generality, we will assume that $$-\pi\not\in X:=\{\theta\in[-\pi,\pi):\,e^{i\theta}\in A\}.$$

Denote $F(t):=\nu\big([-\pi,t)\big)$ for $t\in(-\pi,\pi]$ and $F(-\pi):=0$. Then, see \cite[p.\,35]{Hoffman},
\begin{equation}\label{EQ_fromHerglotz-for-p}
\Re p(re^{i\theta})=\mathcal P_r(\theta+\pi)\|\nu\|+\frac1r\int_{-\pi}^{\pi}K_r(t)\frac{F(\theta+t)-F(\theta-t)}{2\sin t}\,dt,
\end{equation}
where $\|\nu\|:=\nu\big([-\pi,\pi)\big)$ and $K_r(t):=-(1/r)\mathcal P'_r(t)\sin t$ is an approximate identity for~$L^1$ (see, \textit{e.g.}, \cite[p.\,17]{Hoffman}).

From~\eqref{EQ_cont-for-p} and~\eqref{EQ_Herglotz} it follows that $\nu$ has no atoms on $X$. Since $X$ is open, it follows that  $F$ is continuous at every point of~$X$. Bearing this in mind and taking into account that $\mathcal P_r(\theta+\pi)\to 0$ as $(0,1)\ni r\to1$ if $|\theta|<\pi$, from \eqref{EQ_cont-for-p} and~\eqref{EQ_fromHerglotz-for-p} we deduce that $$\liminf_{t\to0+}\frac{F(\theta+t)-F(\theta-t)}{2\sin t}=\liminf_{t\to0+}\frac{\nu\big((\theta-t,\theta+t)\big)}{2t}=0$$ for all $\theta\in X$.  By  \cite[Lemma~1(i) on p.\,37]{Evans-G}, it follows that $\nu(X)=0$. Thus $p$ and hence $G$ extend holomorphically to~$A$. This proves the lemma for the case $\tau\not\in A$.

If $\tau\in A$, we apply the above argument for each of the two arcs forming $A\setminus\{\tau\}$. According to Lemma~\ref{LM_Lind}, $G$ is continuous at~$\tau$ and hence we can apply the Schwarz reflection principle to the function $z\mapsto G(z)/z$ to show that $G$ extends holomorphically to the whole arc~$A$.
\end{proof}

\section{Super-repulsive boundary fixed points}\label{super S}
Boundary regular fixed points of a one-parameter semigroup~$(\phi_t)$ are regular null-points of its generator~$G$ and \textit{vice versa}. For boundary super-repelling fixed points no such results hold. In~\cite[p.\,127]{CDP2} it was pointed out that $\anglim_{z\to x}G(z)=0$ does not imply that~$x$ is a boundary fixed point of~$(\phi_t)$. In fact, neither the converse implication  holds in general, as the following example shows:

\begin{example}
We are going to construct an example of a semigroup $(\phi_t)$ with associated infinitesimal generator $G$ such that $-1$ is a super-repelling boundary fixed point for $(\phi_t)$ but the radial limit of $G$ at $-1$ does not exist.

In order to construct such an example, we will first show that there exists a sequence $\delta_n\in(0,1)$ such that the conformal map $f_0$~of
$$%
\Omega_0:=\{w\colon|\Im w|<1\}\setminus\bigcup\limits_{n\in\Natural}\Big\{w\colon|\Im w|=\delta_n,\, \Re w\le -2n\Big\}
$$%
onto~$\UD$ with the normalization $f_0(0)=0$, $f_0'(0)>0$, satisfies
\begin{align}
\label{EQ_example-liminf}&\liminf_{\Real\ni u\to -\infty}f_0'(u)=0\quad\text{and}\\
\label{EQ_example-limsup}&\limsup_{\Real\ni u\to -\infty}f_0'(u)=+\infty.
\end{align}
The proof of the above statement is based on the following more technical claim.
For $a>0$ and $n\in\Natural$, denote $S(a,n):=\{w\colon |\Im w|<a, \Re w>-2n-1\}$. Furthermore, for a simply connected domain $\Omega\ni0$ let $f_\Omega$ stand for the conformal mapping of $\Omega$ onto~$\UD$
such that $f_\Omega(0)=0$, $f_\Omega'(0)>0$.

\vskip2mm\noindent{\textit{Claim}.} Fix any $a>0$, any $n\in\Natural$, and any $M>0$. Then there exists $\delta=\delta(a,n,M)\in(0,a)$ such that if $\Omega\subset\Complex$ is a simply connected domain satisfying the following two conditions:
\begin{mylist}
\item[(i)] $\Omega$ is symmetric w.r.t. the real axis;

\item[(ii)] $\Omega\cap S(a,n)=S(a,n)\setminus\Big([-2n-1+i\delta,-2n+i\delta]\cup[-2n-1-i\delta,-2n-i\delta]\Big)$,
\end{mylist}
then there exists a point $u_n\in(-2n,-2n+1)$ such that $f_\Omega'(u_n)\ge M$.

\vskip2mm\noindent{\textit{Proof of the Claim.}}
First of all note that by (i),
$f_\Omega(\Real\cap\Omega)=(-1,1)$ and $f'_\Omega(u)>0$ for all $u\in\Real\cap\Omega$.

Now suppose on the contrary that for any
$\delta\in(0,a)$ there exists a simply connected domain $\Omega(\delta)$ satisfying conditions (i) and (ii)  such that
\begin{equation}\label{EQ_example_hypothesis}
f'_{\Omega(\delta)}(u)<M\quad\text{ for all
$u\in(-2n,-2n+1)$ and all $\delta\in(0,a)$.}
\end{equation}
The family $\mathcal
F:=\big(f^{-1}_{\Omega(\delta)}\big)_{\delta\in(0,a)}$ is normal in~$\UD$.
Hence there exists a sequence $(\delta_n)\subset(0,a)$ converging to zero such
that $h_n:=f^{-1}_{\Omega(\delta_n)}$ converges locally uniformly in~$\UD$ to a
holomorphic function~$h_*:\UD\to\ComplexE$. With (ii) taken into account, by the
Carath\'eodory kernel convergence theorem, $h_*$ maps $\UD$ conformally onto a
domain~$\Omega_*$ such that $\Omega_*\cap S(a,n)=S(a,n)\setminus[-2n-1,-2n]$. It follows that
\begin{equation}\label{EQ_contr}
(h_*^{-1})'(u)\to+\infty\qquad \text{as~~~~} (-2n,-2n+1)\ni u\to-2n.
\end{equation}
Again by the
Carath\'eodory kernel convergence theorem, the sequence $h_n^{-1}=f_{\Omega(\delta_n)}$
converges to $h_*^{-1}$ locally uniformly in $\{w\colon |\Im w|<a,\,\Re w>-2n\}$.
Then by~\eqref{EQ_example_hypothesis}, it follows that ${(h_*^{-1})'(u)\le M}$ for all $u\in (-2n,-2n+1)$, which contradicts~\eqref{EQ_contr}. The claim is proved.

\vskip2mm
Take any sequence $(M_n>0)$ tending to~$+\infty$. Now using the above claim define the sequence $(\delta_n)$ recurrently by $\delta_0:=1$ and $\delta_n:=\delta(\delta_{n-1},n,M_n)$ for all~$n\in\Natural$.

Since for each $n\in\Natural$ the domain $\Omega_0$ satisfies conditions~(i) and (ii) with $a:=\delta_n$, it follows that there is a sequence~$(u_n)\subset\Real$ tending to~$-\infty$ such that $f_0'(u_n)\ge M_n$. This proves~\eqref{EQ_example-limsup}.

Notice also that \eqref{EQ_example-liminf} follows from the fact that $f_0(0)-f_0(w)=\int_{w}^0f_0'(u)\,du$ for all $w<0$ and $f_0(w)\to-1$ as $\Real\ni w\to-\infty$.

Now consider the one-parameter semigroup defined by $\phi_t(z):=f_0\big(f_0^{-1}(z)+t\big)$ for all $z\in\UD$ and all~$t\ge0$. The radial limit of the K\oe nigs function $h:=f_0^{-1}$ of~$(\phi_t)$ at $-1$ equals~$\infty$. Consequently, by Proposition~\ref{PR_Pavel}(iii), $-1$ is a boundary fixed point of~$(\phi_t)$.  At the same time, by the above argument, the infinitesimal generator of~$(\phi_t)$ given by $G=f_0'\circ f_0^{-1}$ has no radial limit at~$-1$.
This gives the desired example.
\end{example}

In contrast to the above example, we prove the following proposition, which gives a sufficient condition for an infinitesimal generator to have vanishing angular limit at a super-repelling fixed point of the associated semigroup:

\begin{proposition}\label{PR_superrepulsive}
Let $(\phi_t)$ be a one-parameter semigroup with the associated infinitesimal generator $G$. If $x_0\in\UC$ is a super-repulsive boundary fixed point of ~$(\phi_t)$ and if there exists no backward orbit of~$(\phi_t)$ landing at~$x_0$, then
$$
\anglim_{z\to x_0} G(z)=0.
$$
\end{proposition}

The proof of Proposition \ref{PR_superrepulsive} follows at once from Lemma~\ref{LM_Lind}, Proposition~\ref{PR_contact_arcs}(a2) and the following lemma:

\begin{lemma}\label{LM_no-backward-orbit-means-contact-arc}
Under the hypotheses of Proposition~\ref{PR_superrepulsive}, there exists a maximal contact arc~$A_0$ for which $x_0$ is the initial point.
\end{lemma}
\begin{proof}
First of all, we may assume that the DW-point of~$(\phi_t)$ is~$\tau=1$  thanks to~\cite[Proposition~2.1]{Pavel}, and also that $x_0=-1$.

Let $h$  the K\oe nigs function of~$(\phi_t)$.  By Proposition~\ref{PR_Pavel}(iii),
\begin{equation}\label{EQ_lim-Im-h}
\Im h(-r)\to v_*\quad\text{as~~~~~$[0,1)\ni r\to1$}
\end{equation}
for some $v_*\in\Real$. Now we prove the following claim. \vskip2mm

\noindent\textit{Claim}. For each $r_0\in[0,1)$, $\big[h(-r_0),iv_*+\Re h(-r_0)\big)\subset h(\UD)$.\\[2mm]
Indeed, by Proposition~\ref{PR_Koenigs-function}(B).(ii) and \eqref{EQ_B-P-representation} we have
\begin{equation}\label{EQ_decreasing}
\frac d{dr}\Re h(-r)<0 \qquad\text{for all $r\in [0,1)$.}
\end{equation}
Therefore, from~\eqref{EQ_lim-Im-h} it follows that for any $v_0$ between $v_*$ and $\Im h(-r_0)$ the ray $\{u+iv_0:u\le \Re h(-r_0)\}$ intersect the line $h((-1,0])\subset h(\UD)$, which in view of the translational invariance of~$h(\UD)$ implies our claim.\vskip2mm

Now we show that there is $u_0\in\Real$ such that $R:=\{u+iv_*:u\le u_0\}$ does not intersect~$h(\UD)$. Assume that this is not true.  Then, for some  $u_0\in\Real$, we have $R\subset h(\UD)$ and $\gamma:=h^{-1}(R)$ is a backward orbit of $(\phi_t)$. From the above claim it follows that $R$ and $h((-1,0])$ are equivalent slits in~$h(\UD)$ and hence, see, \textit{e.g.}, \cite[p.\,35\,--\,39]{Goluzin}, $\gamma$ lands at~$-1$, which contradicts the hypothesis.

Now from the claim and from~\eqref{EQ_decreasing} we conclude that the ray $R$ together with the curve $\Gamma:=h((-1,0])\cap\{w\colon\Re w\le u_0\}$ and the part of the line $\{w\colon\Re w=u_0\}$ between $R$ and $\Gamma$ bound a Jordan domain~$D$, which lies in $h(\UD)$. It follows with the help of the Schwarz reflection principle that $h^{-1}$ extends from $D$ to a holomorphic function $f$ on the Jordan domain $D_1:=D\cup R\cup D^*$, where $D^*$ is the reflection of~$D$ w.r.t.~$R$, and that $A:=f(R)$ is a contact arc for~$G$. According to the Carath\'eodory extension theorem (see, \textit{e.g.}, \cite[Thm.\,2.6 on p.\,24]{Pommerenke2}) applied to $f$ in~$D_1$, $$\lim_{R\ni w\to \infty}f(w)=\lim_{w\ni\Gamma\to\infty}f(w)=-1,$$
which means that $x_0=-1$ is the initial point of $A$.

It remains to notice that there exists a unique maximal contact arc $A_0$ containing $A$ and that $-1$ is still the initial point of~$A_0$, because a contact arc cannot contain any boundary fixed points.
\end{proof}

\section{Fractional singularities}\label{S_fractional}

Let $G$ be an infinitesimal generator in $\D$ and let $x\in \de \D$. One can define
\begin{equation*}
\begin{split}
\al_x^+(G)&:=\inf\{\al \in \R: \lim_{r\to 1^-}\frac{|G(rx)|}{(1-r)^\al}=+\infty\}=\sup\{\al\in\R: \liminf_{r\to 1^-}\frac{|G(rx)|}{(1-r)^\al}=0\}\\
\al_x^-(G)&:=\sup\{\al \in \R: \lim_{r\to 1^-}\frac{|G(rx)|}{(1-r)^\al}=0\}=\inf\{\al\in\R: \limsup_{r\to 1^-}\frac{|G(rx)|}{(1-r)^\al}=+\infty\}.
\end{split}
\end{equation*}

\begin{remark}
Let $(z_k)$ be a sequence in $\D$ converging to $x\in \de \D$ non-tangentially and let $p\in\Hol(\D,\C)$ with $\Re p>0$. Then there exists a sequence $(r_k)\subset (0,1)$ such that the hyperbolic distance in $\D$ between $z_k$ and $r_kx$ is less than a constant independent of~$k$. Since $p$, as a map from~$\UD$ to the right half-plane, does not increase the hyperbolic distance, it is not hard to see that $\{\log(|p(z_k)|/|p(r_k x)|)\}_{k\in \N}$ is bounded. From this, with the help of Berkson\,--\,Porta formula~\eqref{EQ_B-P-representation}, one can easily see that it is possible to replace the radial limits in the definition of~$\al_x^\pm (G)$ with the corresponding non-tangential limits.
\end{remark}

The lemma below follows easily from the Berkson\,--\,Porta formula and the growth estimates for holomorphic functions with positive real part (see \cite[eq. (11) p.\,40]{Pommerenke}).

\begin{lemma}\label{LM_alpha_pm}
Let $G$ be an infinitesimal generator in $\D$ and let $x\in \de\D$. If $x$ is not the Denjoy\,--\,Wolff point of the semigroup generated by $G$, then $-1\leq \al_x^-(G)\leq \al_x^+(G)\leq 1$.
\end{lemma}

The examples given below show that, in general, $\al_x^-(G)$ and $\al_x^+(G)$ do not coincide and even if $\al_x^-(G)=\al_x^+(G)=:\alpha$, it is not possible to say anything about $\lim_{r\to 1^-}\frac{|G(rx)|}{(1-r)^{\al}}$.

\begin{example}\label{EX_alpha-plus-minus1}
Using Berkson\,--\,Porta formula~\eqref{EQ_B-P-representation}, consider an infinitesimal generator $G_*(z):=(\tau-z)(1-\overline\tau z)p_*(z)$ for all $z\in\UD$, where $\tau\in\overline\UD\setminus\{1\}$, ${p_*(z):=\tfrac{1+B(z)}{1-B(z)}}$ for all $z\in\UD$ and $B$ is an infinite Blaschke product $B(z):=\prod_{n=1}^{+\infty}\frac{z-x_n}{1-x_n z}$ with all the zeros ${x_n\in(0,1)}$. We claim that $\alpha_1^+(G_*)=1$ and $\alpha_1^-(G_*)=-1$, provided ${\tfrac{a_{n+1}}{a_n}\to+\infty}$ as ${n\to+\infty}$, where, for all $n\in\Natural$, $a_n:=2k_\UD(0,x_n)=\log\tfrac{1+x_n}{1-x_n}$ is (the double of) the hyperbolic distance between $0$ and~$x_n$. Taking into account Lemma~\ref{LM_alpha_pm}, it suffices to show that $(-1)^{j+1}\tfrac{\log |p(z_j)|}{\log|1-z_j|}\to 1$ as ${j\to+\infty}$ for some sequence $(z_j)\subset(0,1)$ converging to~$1$. For each ${j\in\Natural}$, define now ${z_j\in(x_j,x_{j+1})}$ by requiring ${k_\UD(x_j,z_j)=k_\UD(z_j,x_{j+1})}$. Then $(-1)^j B(z_j)\to1$ as ${j\to+\infty}$. Indeed, since $a_{j}-a_{j-1}>2\log(j-1)$ for all $j\in\Natural$ large enough,
\begin{multline*}
\Big|\log|B(z_j)|\Big|~\le~ \sum_{n=1}^{+\infty}\frac{2}{e^{2k_\UD(z_j,x_n)}-1}
~=~\sum_{n=1}^{j-1}\frac{2}{e^{-a_n+(a_{j+1}+a_j)/2}-1}\\~+~\frac{4}{e^{(a_{j+1}-a_j)/2}-1}~+~ \sum_{n=j+2}^{+\infty}\frac{2}{e^{a_n-(a_{j+1}+a_j)/2}-1}~~\to~~0 \quad \text{as~~~$j\to+\infty$.}
\end{multline*}
It follows that
\begin{multline*}
\lim_{j\to+\infty}(-1)^{j+1}\frac{\log |p(z_j)|}{\log|1-z_j|}= \lim_{j\to+\infty}\frac{\log\,\sum_{n=1}^{+\infty} 2e^{-2k_\UD(z_j,x_n)}}{{\log|1-z_j|}}\\ \qquad\qquad\qquad\qquad\qquad\qquad\qquad\qquad\qquad\qquad\quad
=-\lim_{j\to+\infty} \frac{\log\,\sum_{n=1}^{+\infty}e^{-|a_n-(a_j+a_{j+1})/2|}}{(a_j+a_{j+1})/2}\\ =1~-~ \lim_{j\to+\infty}\frac{a_j~+~\log\left(2+\sum_{n=1}^{j-1}e^{a_n-a_j}+\sum_{n=j+2}^{+\infty}e^{a_{j+1}-a_n}\right)} {{(a_j+a_{j+1})/2}}~=~1,
\end{multline*}
and our claim is proved.
\end{example}
\begin{example}
It is now easy to construct an example of an infinitesimal generator~$G$ with any prescribed values $\alpha^+$ and $\alpha^-$ of $\alpha_1^+(G)$ and $\alpha_1^-(G)$, respectively, satisfying the inequality in Lemma~\ref{LM_alpha_pm}. Indeed, with the branch of the power function properly chosen, $$G(z):=(\tau-z)(1-\overline\tau z)\big(p_*(z)\big)^{(\alpha^+-\alpha^-)/2}\left(\frac{1-z}{1+z}\right)^{(\alpha^++\alpha^-)/2}\quad \text{for all~$z\in\UD$,}$$  where $\tau\in\overline\UD\setminus\{1\}$ and $p_*$ is constructed in the previous example, is an infinitesimal generator with the desired property.
\end{example}
\begin{example}
Fix $\alpha\in(-1,1)$ and $\tau\in\overline\UD\setminus\{1\}$. Let
$$G(z):=(\tau-z)(1-\overline\tau z)\left(\frac{1-z}{1+z}\right)^\alpha\big(p_*(F(z))\big)^{1-|\alpha|}\quad\text{for all $z\in\UD$},$$ where $H(z):=(1+z)/(1-z)$, $F(z):=H^{-1}(1+\log H(z))$ for all $z\in\UD$, and $p_*$ is defined in Example~\ref{EX_alpha-plus-minus1}. Then $G$ is an infinitesimal generator, $\al_1^+(G)=\al_1^-(G)=\alpha$, but $$\limsup_{r\to1^-}\frac{|G(r)|}{(1-r)^\al}=+\infty\quad\text{and}\quad
\liminf_{r\to1^-}\frac{|G(r)|}{(1-r)^\al}=0.$$
\end{example}

Many other similar ``pathological'' examples can be found. In this paper we consider ``regular'' fractional singularities defined as follows:

\begin{definition}\label{DF_regular-singularity}
Let $\alpha\in\R\setminus\{0\}$. An infinitesimal generator~$G$ is said to have a \textsl{regular singularity of order~$\alpha$} at a point~$x\in\UC$ if the angular limit
\begin{equation}\label{EQ_limit-in-def-reg-sing}
M_\al(x):=\lim_{r\to 1^-}\frac{G(rx)}{(1-r)^\alpha}
\end{equation}
exists and belongs to~$\Complex^*:=\Complex\setminus\{0\}$.
\end{definition}

\begin{lemma}\label{LM_ang_lim_at_frac_sing}
Let $G$ be an infinitesimal generator. If $G$ has a regular singularity of order~$\alpha$ at $x\in\de \D$ then
\begin{equation}\label{EQ_ang_lim_at_frac_sing}
\anglim_{z\to x}\frac{G(z)}{(1-\overline{x}z)^\alpha}=M_\al(x).
\end{equation}
\end{lemma}

\begin{proof} Thanks to Berkson\,--\,Porta formula~\eqref{EQ_B-P-representation} and \cite[Theorem 4.3 p.76]{Pommerenke2} it is enough to show that if $p:\D\to \C$ is holomorphic and $\Re p>0$, then $f(z):=\log (p(z)/(1-\overline{x}z)^\beta)$ is a Bloch function for any $\beta$, {\it i.e.} \[
\sup_{z\in \D} |f'(z)|(1-|z|^2)<+\infty.
\]
This is equivalent to proving that $(1-|z|^2)|p'(z)|/|p(z)|$ is bounded in $\D$. Let $p_0(z):=\frac{1+z}{1-z}$. Then $\omega(z):=p_0^{-1}\circ p(z)$ is a holomorphic self-map of $\D$. Hence using the Schwarz\,--\,Pick inequality we get
\[
\frac{|p'(z)|}{|p(z)|}=\frac{2|\omega'(z)|}{|1-\omega(z)^2|}\leq \frac{2}{1-|z|^2},
\]
and the desired result follows.
\end{proof}

\begin{remark}\label{RM_alphas_coincide_for_reg_sing}
If $x$ is a regular singularity of~$G$ of order~$\alpha$ different from
the Denjoy\,--\,Wolff point of the associated one-parameter semigroup, then $\alpha^+(G)=\alpha^-(G)=\alpha$, hence $\al\in [-1,1]\setminus\{0\}$. If $\alpha=-1$, then $x$ is a boundary regular pole, and if $\alpha=1$ then $x$ is a boundary regular null point of~$G$.

Similarly, if $G$ has a regular singularity of order~$\alpha$ at the Denjoy\,--\,Wolff point, then  $\alpha\in[1, 3]$. This case was deeply analyzed in~\cite{Elin-Khavinson-Reich-Shoikhet}. So from now on we will mainly consider regular singularities of order $\alpha\in[-1, 1)\setminus\{0\}$, hence different from the Denjoy\,--\,Wolff point.
\end{remark}

In the case of boundary regular null points and regular poles the limit~\eqref{EQ_limit-in-def-reg-sing} in Definition~\ref{DF_regular-singularity} can be replaced with the radial or angular limit of $|G(z)|/|1-z|^\alpha$ at~$x$. However, for $\al\in (-1,1)$ these three limits are pairwise non-equivalent, as the following two examples show.

\begin{example}\label{EX_rad-of-module-but-not-angle-lim}
Let $\alpha\in(-1,1)\setminus\{0\}$ and $a>0$. Consider $G(z):=-zp(z)$ with $p(z):=[(1+z)/(1-z)]^{\alpha}\exp(i\sin f(z))$ for $z\in\UD$, where $f$ is the conformal mapping of~$\UD$ onto ${\{w:|\Im w|<a\}}$ with $f(0)=0$, $f'(0)>0$, and the branch of the power functions is chosen in such a way that $p(0)=1$. If $\sinh a\le\pi(1-|\alpha|)/2$, then ${\Re p>0}$ and hence $G$ is an infinitesimal generator. It is easy to see that ${\lim_{(0,1)\ni r\to1}|G(r)|/(1-r)^{\alpha}}=2^{\alpha}$ and hence $\al_1^-(G)=\al_1^+(G)=\al$, while the angular limit $\anglim_{z\to 1}|G(z)|/|1-z|^{\alpha}$ does not exist. In particular, by Lemma~\ref{LM_ang_lim_at_frac_sing} and Remark~\ref{RM_alphas_coincide_for_reg_sing}, the point $1$ is not a {\sl regular} singularity of any order $\al'\in [-1,1]\setminus\{0\}$.
\end{example}

\begin{example}\label{EX_anglim-of-modulus-but-not-anglim-itself}
Similar to the previous example we consider an infinitesimal generator of the form $G(z):=-zp(z)$ with $p(z):=[(1+z)/(1-z)]^{\alpha}\exp g(z)$, where $|\Im g(z)|<\omega:=\pi(1-|\alpha|)/2$ for all~$z\in\UD$. If $\anglim_{z\to 1}\Re g(z)$ exists finitely, then ${\anglim_{z\to 1}|G(z)|/|1-z|^{\alpha}}$ exists. At the same time, if $\anglim_{z\to 1}\Im g(z)$ does not exist, then \eqref{EQ_ang_lim_at_frac_sing} does not  exist either and hence, by Lemma~\ref{LM_ang_lim_at_frac_sing}, again~$1$ is not a regular singularity of~$G$. Examples of such univalent functions $g$ can be constructed by means of the theory of prime ends, see, \textit{e.g.}, \cite[Chapter\,9]{ClusterSets}: consider a simply connected domain ${D\subset\omega\UD}$  having a prime end~$P$ of the fourth kind with the set of all principal points forming a non-degenerate segment parallel to the imaginary axis, and define $g$ to be any conformal map of~$\UD$ onto~$D$ under which~$1$ corresponds to the prime end~$P$.
\end{example}

Now we prove our theorem about the characterization of fractional order regular singularities of infinitesimal generators in terms of the boundary behaviour of the associated one-parameter semigroups and K\oe nigs function.

\begin{proof}[Proof of Theorem \ref{TH_fractional}]
Assume first that (i) holds and let us prove (ii). Note that $\int_0^1 dr/G(rx)$ converges. According to Proposition~\ref{PR_Koenigs-function}, this implies that $h(x)\in\Complex$. Therefore, by Proposition~\ref{PR_Pavel}(iii), $x$ is not a boundary fixed point of $(\phi_t)$.
Further, differentiating the identity $\phi_t\circ \phi_s=\phi_{t+s}$ in $s$ and taking $s=0$,  we obtain
\begin{equation}\label{EQ_diff}
\left[\phi_t'(z)(1-\overline{x}z)^\al\right]\frac{G(z)}{(1-\overline{x}z)^\al}=G(\phi_t(z))\quad\text{for all $z\in \D$}.
\end{equation}

If $\phi_t(x)\in \D$ for all $t>0$,  then $G(\phi_t(x))\neq 0$ for all $t>0$, because $\phi_t(x)$ cannot coincide with the Denjoy\,--\,Wolff point of $(\phi_t)$. Indeed, taking into account univalence of~$\phi_t$ in~$\UD$, we have $\phi_t(x):=\anglim_{z\to x}\phi_t(x)\in\partial\phi_t(\UD)$ for all~$t\ge0$.  Therefore, in this case,  with the help of Lemma~\ref{LM_ang_lim_at_frac_sing} we can pass in~\eqref{EQ_diff}  to the angular limit as $z\to x$, which immediately leads to~(ii) for all $t>0$.

If $x$ is a contact point of $\phi_t$ for some $t>0$, then $x$ is a contact point of~$\phi_t$ for all $t\in (0,t_1)$, where $t_1:=\sup\{t\ge0\colon\phi_t(x)\in\UC\}\in(0,+\infty]$. By Step~1 in the proof of Theorem~\ref{TH_contact-points}, $G$ extends holomorphically on the arc $A:=\{\phi_t(x)\colon t\in(0,t_1)\}$ and does not vanish on~$A$. Therefore, again \eqref{EQ_diff} together with Lemma~\ref{LM_ang_lim_at_frac_sing} implies (ii) for all $t\in(0,+\infty)\setminus\{t_1\}$.

Clearly (ii) implies (iii). So now we assume (iii) and let us prove (i). First consider the case $\al<0$ and let $t_0>0$ be such that the angular limit~\eqref{EQ_fractional-phi-prime} exists and belongs to~$\Complex^*$. Then $\anglim_{z\to x}\phi_{t_0}'(z)=0$. By \cite[Lemma 3.8]{BCD-M_regular-poles} it follows that $\phi_{t_0}(x)\in \D$. Since $\phi_{t_0}(x)$ cannot be the Denjoy\,--\,Wolff point of $(\phi_t)$ as we pointed out before, substituting $t:=t_0$ in \eqref{EQ_diff} and passing to the radial limit as $z\to x$, we obtain (i).

Consider now the case $\al>0$. Arguing as above, we see that to prove~(i) it is sufficient to show that~$x$ is not a boundary fixed point of~$(\phi_t)$. Suppose on the contrary that $x$ \textsl{is} a boundary fixed point of~$(\phi_t)$. Let $t_1>t_0>0$ be the two values for which the angular limit \eqref{EQ_fractional-phi-prime} exists and belongs to~$\Complex^*$. Note that $\phi'_{t_1}(z)=\phi'_{t_1-t_0}(\phi_{t_0}(z))\cdot \phi'_{t_0}(z)$ for all~$z\in\UD$. Multiplying this identity by $(1-\overline{x}z)^\al$ on both sides and passing to the radial limit as~$z\to x$, we conclude that there exists $\lim_{r\to 1^-}\phi'_{t_1-t_0}(\phi_{t_0}(rx))=:C\in \C^\ast$. By Lindel\"of's theorem it follows that $\anglim_{z\to x}\phi'_{t_1-t_0}(z)=C\in \C^\ast$ and hence, by \cite[Lemmas\,1\,and\,3]{CDP2}, $x$~is a boundary \textsl{regular} fixed point of~$\phi_t$ for all~$t\ge0$, which clearly contradicts~(iii).

Next, in case the Denjoy\,--\,Wolff point $\tau\in\partial\UD$ in order to prove the equivalence between (i) and~(iv) it is clearly enough to apply Lemma~\ref{LM_ang_lim_at_frac_sing} and Proposition~\ref{PR_Koenigs-function}(B).(ii), while in case~$\tau\in\UD$ one should first pass to the lifting of~$(\phi_t)$ given by~\cite[Proposition\,\,2.1]{Pavel}.

Finally, by \cite[Theorem\,\,1]{Starkov}, (v) implies~(iv). To prove the converse implication it is enough to apply the mean value theorem for the real and imaginary parts of~$h$ along all intervals in $\D$ ending at $x$. The proof is now complete.
\end{proof}

We end this section with a few more comments and examples.
\begin{example}
Let $\alpha\in(0,1)$. With the help of Berkson\,--\,Porta formula~\eqref{EQ_B-P-representation} it is easy to see that $G(z):=-i (1-z)^2 \left(\tfrac{i-z}{1-iz}\right)^\alpha$ for all $z\in\UD$, where the branch of the power function is chosen in such a way that $G(0)=-ie^{i\pi\alpha/2}$, is an infinitesimal generator. Note also that $A_0:=\{e^{i\theta}\colon \theta\in(\pi/2,3\pi/2)\}$ is a maximal contact arc of~$G$ and its initial point $x_0:=i$ is a regular singularity of order~$\alpha$ for~$G$, while its final point~$x_1:=-i$ is a regular singularity of order~$-\alpha$. In particular, by Theorem~\ref{TH_fractional}, $x_0$ and $x_1$ are not boundary fixed points of the associated one-parameter semigroup~$(\phi_t)$. Therefore, by the argument in the proof of Proposition~\ref{PR_contact_arcs}(iii), there exists~$t_1$ such that~$\phi_{t_1}(x_0)=x_1$. From~\eqref{EQ_diff} it follows that for $x:=x_0$ and $t:=t_1$ the limit~\eqref{EQ_fractional-phi-prime}  equals~$\infty$. Thus, this example shows that for~$\alpha\in(0,1)$, Theorem~\ref{TH_fractional}  would fail without allowing a possible exceptional value of~$t$ in assertion~(ii).
\end{example}
\begin{remark}
Note that in contrast to the above example, for $\alpha\in(-1,0)$ we can require in Theorem~\ref{TH_fractional}(ii) that~\eqref{EQ_fractional-phi-prime} exists and belongs to~$\Complex^*$  \textsl{for all}~$t>0$ without permitting any exceptional values. This follows directly from the proof of~Theorem~\ref{TH_fractional}.
\end{remark}
\begin{remark}
For~$\alpha=-1$, Theorem~\ref{TH_fractional} was already proved in~\cite{BCD-M_regular-poles}. The proof given in that paper does not rely on the study of contact points exploited in the present proof.

At the same time, according to~\cite[Theorem~1]{CDP2},  Theorem~\ref{TH_fractional} becomes false if~$\alpha=1$, which corresponds to the case of a regular null-point of~$G$. In fact, it is known (see, \textit{e.g.}, \cite[Prop.\,4.8 p.\,80]{Pommerenke2}) that since $\phi_t$ has a finite angular limit at~$x$, we have $(z-x)\phi'_{t}(z)\to0$ as $z\to x$ non-tangentially, and hence (ii) and~(iii) are never satisfied with~$\alpha=1$.

Finally, if one had to consider ``fractional singularities of order $\alpha=0$'' defined as before, one could check that assertion (i) in Theorem~\ref{TH_fractional} implies~(ii), but the converse implication is not true. For instance, if $x$ is any boundary regular fixed point of~$(\phi_t)$, then (ii) holds by the very definition, but by~\cite[Theorem~1]{CDP2}, $G$ has a regular singularity of order~$\alpha=1$ at~$x$.
\end{remark}

To conclude, we give an example showing that for $\alpha>0$ it is not possible to manage with only one value of~$t$ in assertion~(iii) of~Theorem~\ref{TH_fractional}.
\begin{example}
Let $\phi_t(z):=H^{-1}\big(H(z)^{\exp(-t)}\big)$ for all $z\in\UD$ and all~$t\ge0$, where $H(z):=(1+z)/(1-z)$ and the branch of the power function is chosen in such away that $\phi_t(0)=0$ for all~$t\ge0$. It is easy to see that $(\phi_t)$ is a one-parameter semigroup such that for every $t>0$, limit~\eqref{EQ_fractional-phi-prime} exists and belongs to~$\Complex^*$ for $x:=-1$ and $\alpha:=1-e^{-t}$, while the associated infinitesimal generator satisfies $G(z)[(1+z)\log(1+z)]^{-1}\to-1$ as $\UD\ni z\to-1$.
\end{example}

\section{Geometric conditions for fractional singularities}\label{S geometric}

The problem of characterization of boundary regular singularities of infinitesimal generators via geometric properties of the image domains of the corresponding K\oe nigs maps is closely related to the problem of existence of the non-vanishing angular derivative of a conformal map at a boundary point. Up to our best knowledge, the only known geometric characterization in the latter problem, see, \textit{e.g.}, \cite[Theorem A]{RodinWarschawski}, is given in terms of the extremal length. In concrete cases, it is very difficult to apply this characterization directly. Below we use a consequence of this characterization \cite[Theorem~5]{RodinWarschawski} to prove  quite general sufficient and necessary conditions for boundary regular singularities of order~$\alpha\in(-1,1]\setminus\{0\}$.

\subsection{Sufficient criteria} In order to formulate the sufficient conditions we need to introduce two notions.
Since we mainly deal with non-tangential rather than unrestricted limits at boundary points, the notion of a Dini-smooth corner suitable for our aim is as follows.
\begin{definition}
We say that a Jordan curve $\Gamma\subset\Complex$ has a \textsl{Dini-smooth corner of opening ${\theta\in[0,2\pi]}$ at a point~$w_0\in\Gamma$} if the following conditions are satisfied:
\begin{mylist}
\item[(i)] for all $\rho>0$ small enough $\Gamma$ intersects $C_\rho:=\{w:|w-w_0|=\rho\}$ at exactly two  points $\gamma^+(\rho)$ and $\gamma^-(\rho)$, and $\rho\mapsto\gamma^\pm(\rho)$ are continuous functions,

\item[(ii)]  the functions $\rho\mapsto\arg\big(\gamma^\pm(\rho)-w_0\big)$ have Dini-continuous extensions to $\rho=0$, and

\item[(iii)] the angle formed by the one-sided tangents to~$\Gamma$ at~$w_0$, interior w.r.t. the Jordan domain bounded by~$\Gamma$,  is of magnitude~$\theta$.
\end{mylist}
\end{definition}
Note that this definition differs substantially from the one given in \cite[\S3.4]{Pommerenke2}.
One can show that for any conformal mapping $f$ of $\UD$ onto a Jordan domain whose boundary has a Dini-smooth corner of opening~$\alpha\,\pi$, $\alpha\in(0,2]$, at a point $f(x)$, $x\in\UC$, there exists nonzero finite angular limit $\anglim_{z\to x}(f(z)-f(x))/(z-x)^\alpha$. To extend such kind of statements to a much wider class of domains we introduce the following definition.
\begin{definition}
Let $\Gamma\subset\Complex$ be a Jordan curve and $w_0\in\Gamma$. A set $E\subset \Gamma$ is said to be \textsl{locally dense on~$\Gamma$ at the point~$w_0$} if for some (and hence every)  $w_1\in\Gamma\setminus\{w_0\}$ each of the two connected components of~$\Gamma\setminus\{w_0,w_1\}$ contains a sequence $(w_n)\subset E$ such that $|w_n-w_0|$ converges monotonically to~$0$ and $$\sum_{n\in\Natural}\left(\log\frac{|w_{n\hphantom{+}}-_{\hphantom{1}}w_0|}{|w_{n+1}-w_0|}\right)^2<+\infty.$$
\end{definition}

Now we can give the first sufficient criteria for fractional order regular singularities:

\begin{theorem}\label{TH_geometr-charact-alpha_neq_1}
Let $(\phi_t)$ be a one-parameter semigroup with associated infinitesimal generator $G$ and K\oe nigs function $h$. Let $x\in\UC$, and let $\alpha\in(-1,1)\setminus\{0\}$. Suppose that there exist Jordan domains~$\Omega_0$ and $\Omega_1$ (possibly $\Omega_0=\Omega_1$) with $w_0:=h(x)\subset\partial \Omega_0\cap\partial\Omega_1$ such that
\begin{mylist}
\item[(i)] $h([rx,x))\subset \Omega_0\cap\Omega_1$ for some $r\in(0,1)$;
\item[(ii)] the Jordan curves $\partial\Omega_j$, $j=0,1$, have Dini-smooth corners of opening~$(1-\alpha)\pi$~at~$w_0$;
\item[(iii)] $\Omega_0\subset h(\UD)$ and the set $\partial \Omega_1\setminus h(\UD)$ is locally dense on~$\partial \Omega_1$ at~$w_0$.
\end{mylist}
Then $G$ has a boundary regular singularity of order~$\alpha$ at~$x$.
\end{theorem}

\begin{proof}
Fix $\alpha\in(-1,1)\setminus\{0\}$.  We have to prove that $x$ is a regular singularity of order~$\alpha$ for~$G$. By Theorem~\ref{TH_fractional} it is sufficient to show that the following limit exists:
\begin{equation}\label{EQ_anglim-frac-h}
\anglim_{z\to x}\frac{h(z)-h(x)\hphantom{{}^{1-\alpha}}}{(1-\overline x\,z)^{1-\alpha}}\in\Complex^*.
\end{equation}
Moreover, notice that $h(\UD)\cap S=\emptyset$, where $S$ is a ray in~$\Complex$ or a part of a logarithmic spiral beginning from~$w_0=h(x)$ and tending to~$\infty$. Choose any single-valued branch~$F$ of~$w\mapsto iC_0-((1-\alpha)\pi)^{-1}\log(w-w_0)$ in $\Complex\setminus S$, where $C_0$ is a real constant to be determined later, and consider the univalent function $$f:=(F\circ h\circ q)^{-1},\quad\text{where~} q(z):=x\frac{e^{\pi z}-1}{e^{\pi z}+1}~~\text{~for~all~$|\Im z|<1/2$},$$
that maps $F(h(\UD))$ onto the strip $\Pi:=\{z\colon \,|\Im z|<1/2\}$. In this setting,~\eqref{EQ_anglim-frac-h} is equivalent to the existence of a finite limit of $f^{-1}(z)-z$ as $\Re z \to+\infty$, $|\Im z|<a$, for all $a\in(0,1/2)$.

By hypothesis, $\partial\Omega_0$ has a Dini-smooth corner of opening~$(1-\alpha)\pi$ at~$w_0$. Therefore, choosing a suitable value of the constant $C_0$, we have that for any $a\in(0,1/2)$ there exists $u=u(a)>0$ such that $\{\zeta\colon\,\Re \zeta>u,\,|\Im\zeta|<a\}\subset F(\Omega_0)$. Thus, our task is reduced to proving that {\sl $F(h(\UD))$ has angular derivative at~$+\infty$ in the sense of Rodin\,--\,Warschawski} (see~\cite[Section~4]{RodinWarschawski}).

\Claim A{the domain $F(\Omega_0)$ has  angular derivative at~$+\infty$ (in the sense of Rodin\,--\,Warschawski \cite[Section~4]{RodinWarschawski})}
From the fact that $\partial\Omega_0$ has a Dini-smooth corner at~$w_0$, it follows that for all $u>0$ large enough, say for $u\ge u_0$, the intersection of $F(\Omega_0)$ with the line $\{\zeta\colon \Re\zeta=u\}$ is a segment $\big(u+i\nu_0'(u),u+i\nu_0''(u)\big)$, $\nu_0'(u)<\nu_0''(u)$, and that the function $$\omega(u):=\sup\{|\nu_0'(\tilde u)+\tfrac12|,|\nu_0''(\tilde u)-\tfrac12|\colon\,\tilde u\ge u\},\qquad u\ge u_0,$$
is continuous, non-increasing and integrable on $[u_0,+\infty)$. We may also assume that $\omega(u)>0$ for all $u\ge u_0$, because otherwise our claim is trivial.
Using these properties of~$\omega$, it is easy construct inductively a sequence $(u_n)\subset(u_0,+\infty)$ tending monotonically to~$+\infty$ such that $u_n-u_{n-1}=\omega(u_n)$ for all~$n\in\Natural$. Then we have
$$
\sum_{n\in\Natural}(u_n-u_{n-1})^2=\sum_{n\in\Natural} \omega(u_n)^2=\sum_{n\in\Natural}(u_n-u_{n-1})\omega(u_n)\le\int_{u_0}^{+\infty} \omega(u)\,du<+\infty.
$$
Thus $\{u_n\}$ is a so-called {\sl subdivision} of $F(\Omega_0)$ (see \cite[Section~2]{RodinWarschawski}) such that the sequences $\delta_n:=u_{n+1}-u_n$, $\theta_n':=\tfrac12+\sup\{\nu_0'(u)\colon u\in[u_n,u_{n+1}]\}$, and $\theta_n'':=\tfrac12-\inf\{\nu_0''(u)\colon u\in[u_n,u_{n+1}]\}$ are square-summable. Thus Claim A follows directly from \cite[Theorem\,\,5]{RodinWarschawski}. \vskip2mm

Now, denote by $F_1$ the holomorphic extension of $F$ from the connected component of~$\Omega_1\setminus S$ containing~$h([rx,x))$ to the whole~$\Omega_1$. Fix any point $\zeta_1\in\partial F_1(\Omega_1)$ different from~$\infty$. By hypothesis, it follows that the components of $\partial F_1(\Omega_1)\setminus\{\zeta_1,\infty\}$ are two Jordan arcs $\Gamma^+$ and $\Gamma^-$ approaching~$\infty$ along the two asymptotes $\{u+i/2\colon u>0\}$ and $\{u-i/2\colon u>0\}$, respectively. Therefore, the argument in the proof of Claim~A can be repeated with the function~$\omega$ replaced by
$$\omega_1(u):=\sup\{|\nu'_0(\tilde u)+\tfrac12|,|\nu''_0(\tilde u)-\tfrac12|,|\nu_1'(\tilde u)+\tfrac12|,|\nu_1''(\tilde u)-\tfrac12|\colon\,\tilde u\ge u\},$$ where  $\nu_1'(u)$ and $\nu_1''(u)$ are defined for the domain $F_1(\Omega_1)$ in the same way as $\nu_0'(u)$ and $\nu_0''(u)$ were defined for $F(\Omega_0)$. For convenience we will keep the same notation~$(u_n)$ for the new subdivision of~$F(\Omega_0)$ and~$F_1(\Omega_1)$ constructed in this way.  In particular, we have
\begin{equation}\label{EQ_omega_1^2-converges}
\sum_{n\in\Natural}\omega_1(u_n)^2<+\infty.
\end{equation}\vskip2mm

\Claim B{There is a subsequence~$(u^*_n)$ of $(u_n)$ such that $\sum_{n\in\Natural}(u^*_{n+1}-u^*_n)^2<+\infty$ and for each $n\in\Natural$ the strip $\{\zeta\colon u^*_n\le\Re\zeta\le u^*_{n+1}\}$ intersects both $\Gamma^+\setminus F(h(\UD))$ and~$\Gamma^-\setminus F(h(\UD))$}

By condition~(iii), there exist two sequences $(\xi_n^\pm)\subset\Gamma^\pm\setminus F(h(\UD))$ such that $a^\pm_n:=\Re\xi^\pm_n$ tend monotonically to~$+\infty$ and the series $\sum(a_{n+1}^\pm-a_{n}^\pm)^2$ converge. For $x\in\Real$ let us denote by $a^\pm(x)$ the element of $\{a^\pm_n:n\in \Natural\}\cap(x,+\infty)$ closest to~$x$. Similarly, we define the function $u(x)$ for the sequence $(u_n)$. Now we construct~$(u_n^*)$ recurrently in the following way. First we put $q_1:=\max\{a_1^+,a_1^-\}$. Then for each $n\in\Natural$ we let $u_{n}^*:=u(q_n)$ and $q_{n+1}:=\max\{a^+(u_{n}^*), a^-(u_{n}^*)\}$.

By the very construction, $(u_n^*)$ is a subsequence of~$(u_n)$ and each segment $[u_n^*,u_{n+1}^*]$ contains at least one of $a^+_j$'s and one of~$a^-_j$'s. Again by construction, for every $n\in\Natural$, the interval $I'_n:=(u^*_n,q_{n+1})$ can contain points of only one of the sequences $(a^+_n)$ or $(a^-_n)$. Hence $I_n'$ lies between two consecutive elements of~$(a^+_n)$ or  between two consecutive elements of~$(a^-_n)$. Therefore, $\sum_{n\in\Natural}(q_{n+1}-u^*_n)^2<+\infty$. Similarly, $I_n'':=(q_{n+1}, u^*_{n+1})$ lies between two consecutive elements of~$(u_n)$ and hence $\sum_{n\in\Natural}(u^*_{n+1}-q_{n+1})^2<+\infty$. It follows that $\sum_{n\in\Natural}(u^*_{n+1}-u^*_{n})^2<+\infty$ and thus Claim~B is proved.\vskip2mm

Now let $J(u)$ stand for the connected component of~$F(h(\UD))\cap\{\zeta\colon\Re\zeta=u\}$ that intersects~$\Real$ and denote $$v'_n:=~\sup_{\mathclap{\,u^*_{n}\le u\le u^*_{n+1}}\,}~\inf\{\Im\zeta\colon\zeta\in J(u)\}, \quad v''_n:=~\inf_{\mathclap{\,u^*_{n}\le u\le u^*_{n+1}}\,}~\sup\{\Im\zeta\colon\zeta\in J(u)\}.$$
Note that, on the one hand, $\big[u,u+i\nu_0''(u)\big)\subset J(u)\cap\{\zeta\colon\Im\zeta>0\}$ for all~$u\ge u_0$, while on the other hand, $J(u)\cap\{\zeta\colon\Im\zeta>0\}\subset \big[u,u+i\nu_1''(u)\big)$ holds for all $u\in\{a_n^+:\,n\in\Natural\}$. Therefore, $|v''_n-\tfrac12|\le\omega_1(u^*_n)$ for all $n\in\Natural$.
Since $(u_n^*)$ is a subsequence of $(u_n)$, from~\eqref{EQ_omega_1^2-converges} we get
$$
\sum_{n\in\Natural}(v''_n-\tfrac12)^2\le \sum_{n\in\Natural}\omega_1(u_n^*)^2\le \sum_{n\in\Natural}\omega_1(u_n)^2<+\infty.
$$
Analogously, $\sum_{n\in\Natural}(v'_n+\tfrac12)^2<+\infty$. Recall finally that also $\sum(u^*_{n+1}-u^*_n)^2$ converges. Thus by \cite[Theorem~5]{RodinWarschawski}, $F(h(\UD))$ has an angular derivative at~$+\infty$ in the sense of Rodin\,--\,Warschawski, and we are done.
\end{proof}

\begin{remark}
The above result holds also for $\alpha=0$, implying that under the hypotheses of Theorem~\ref{TH_geometr-charact-alpha_neq_1} with $\alpha=0$, the infinitesimal generator~$G$ has finite nonzero angular limit at $x$.
\end{remark}

For the case $\alpha=-1$,  as a ``byproduct" of the proof of Theorem~\ref{TH_geometr-charact-alpha_neq_1} we obtain the following simpler criterion:

\begin{theorem}\label{TH_geometr-charact-alpha=1}
Let $(\phi_t)$ be a one-parameter semigroup with associated infinitesimal generator $G$ and K\oe nigs function $h$. Let $x\in\UC$, and suppose that there exists a Jordan domain~$\Omega_0\subset h(\UD)$ with $w_0:=h(x)\subset\partial \Omega_0$ such that $h([rx,x))\subset \Omega_0$ for some $r\in(0,1)$ and $\partial\Omega_0$ has a Dini-smooth corner of opening~$2\pi$ at $w_0$.
Then $G$ has a boundary regular pole at~$x$.
\end{theorem}

\begin{proof}
Let $\alpha:=-1$.  Arguing as in the proof of Theorem~\ref{TH_geometr-charact-alpha_neq_1}, we see that the domains $F(\Omega_0)$ and $F(\C\setminus S)$ both have angular derivative at~$+\infty$ in the sense of Rodin\,--\,Warschawski. Since  $\Omega_0\subset h(\UD)\subset\Complex\setminus S$, this implies (use, \textit{e.g.}, \cite[Theorem~A]{RodinWarschawski}) that $F(h(\UD))$ has also angular derivative at~$+\infty$ in the sense of Rodin\,--\,Warschawski, which proves the theorem.
\end{proof}

\begin{remark}
Theorem \ref{TH_geometr-charact-alpha=1}  can be also proved using~\cite[Corollary 11.11, p.\,261--262]{Pommerenke2}, or as a consequence
of~\cite[Corollary 2.36, p.\,36]{Bertilsson} if one takes into account the relation between $\beta$-points of K\oe nigs functions and boundary regular poles of infinitesimal generators \cite[Theorem~1.1]{BCD-M_regular-poles} (see also the next section).
\end{remark}

\subsection{Necessary criteria} We first give an easy necessary condition for fractional order regular singularities, analogous to the geometrical characterization of isogonality (see, \textit{e.g.}, \cite[Theorem\,\,11.6, p.\,254]{Pommerenke2}):

\begin{theorem}\label{TH_necessary-frac-singularity}
Let $(\phi_t)$ be a one-parameter semigroup with associated infinitesimal generator $G$  and  K\oe nigs function $h$. Suppose that $x\in\UC$ is a boundary regular singularity of order~${\alpha\in[-1,1)\setminus\{0\}}$. Then the following assertions hold:
\begin{mylist}
\item[(i)] for any $\theta\in\big(0,(1-\alpha)\pi\big)$ there exists $\rho>0$ such that $S(\theta,\rho):=\{\nu\varrho e^{i\vartheta}\colon{|\vartheta|<\theta/2,\,}\\ {0<\varrho<\rho}\}\subset h(\UD)$, where $\nu:=\lim_{r\to1-}\big(h(x)-h(rx)\big)\big|h(x)-h(rx)\big|^{-1}$;

\item[(ii)] for any $\theta>(1-\alpha)\pi$ and any $\rho>0$, $S(\theta,\rho)\not\subset h(\UD)$.
\end{mylist}
\end{theorem}

\begin{proof}
First of all note that $\int_0^1 |G(rx)|^{-1}\,dr<+\infty$. Hence by Proposition~\ref{PR_Koenigs-function}, $h(x)$ is finite.
Construct the maps $F$ and $g:=f^{-1}=F\circ h\circ q$ as in the proof of Theorem~\ref{TH_geometr-charact-alpha_neq_1}. From the fact that $x$ is a boundary regular singularity of order~$\alpha$ for $G$ it is not difficult to deduce that $$\beta:=\lim_{\substack{\Re z\to+\infty,\\|\Im z|<a}} g(z)-z$$ exists finitely for any $a\in(0,1/2)$. Choosing an appropriate value of~$C_0$ in the definition of~$F$ we may assume that $\beta\in\Real$. Then it follows that for any $a\in(0,1/2)$ there exists $u=u(a)$ such that $\{\zeta\colon\,\Re \zeta>u,\,|\Im\zeta|<a\}\subset F(h(\UD))$ and that $$\sup\big\{|\Im\zeta|\colon \zeta\in F\big(h([0,x))\big),\,\Re\zeta>u\big\}\to0\quad\text{as~}u\to+\infty.$$
This proves assertion~(i). To prove~(ii) we note that otherwise  $\{\zeta\colon {\Re \zeta>u},{|\Im\zeta|<a}\}\subset F(h(\UD))$ for some ${a>1/2}$ and $u>0$, which would contradict \cite[Proposition~4]{RodinWarschawski} applied for $R:=F(h(\UD))$, $u_n:=\log n$.
\end{proof}

\begin{remark}
Again, the above result holds also for $\alpha=0$. Namely, if $G$ admits finite nonzero angular limit at $x$, then the conclusion of Theorem \ref{TH_necessary-frac-singularity} holds with~$\alpha=0$.
\end{remark}

In case $\alpha=-1$, that is for regular poles, we give a necessary and sufficient criterion in terms of Bertilsson's condition. Bertilsson's condition~\cite[Corollary 2.36, p.\,36]{Bertilsson} is a  sufficient (but in general not necessary) condition  for a conformal map to have a $\beta$-point:
\begin{theorem}[Bertilsson]\label{Bert}
Let $h$ be a conformal map of ~$\D$.  Let $x\in\UC$ and suppose that $w_0:=\anglim_{z\to x}h(z)$ exists finitely. Fix $\gamma>1$. Denote by $\alpha_k$, $k\in\Natural$, the opening of the smallest angle with vertex at~$w_0$, containing the set ${\{w\not\in h(\UD)\colon \gamma^{-k-1}\le|w|\le\gamma^{-k}\}}$. If
\begin{equation}\label{EQ_Bertilsson-condition}
\sum_{n\in\Natural}\alpha_k<+\infty,
\end{equation}
then $x$ is a $\beta$-point of $h$.
\end{theorem}
Although  not true for general conformal mappings, in case of K\oe nigs functions we prove that Bertilsson's condition is also necessary:
\begin{proposition}\label{PR_Bertilsson}
Let $(\phi_t)$ be a one-parameter semigroup with associated infinitesimal generator $G$ and K\oe nigs function $h$. A point $x\in \de \D$ is a regular boundary pole of~$G$ if and only if~\eqref{EQ_Bertilsson-condition} holds.
\end{proposition}

\begin{proof}
Condition~\eqref{EQ_Bertilsson-condition} is sufficient for $x$ to be a boundary regular pole of~$G$ due to Theorem~\ref{Bert} and~\cite[Theorem~1.1]{BCD-M_regular-poles}. It remains to show that~\eqref{EQ_Bertilsson-condition} is also necessary.

So assume that~$x$ is a boundary regular pole of~$G$. Arguing as in the proof of Theorem~\ref{TH_necessary-frac-singularity}, we see that $h(x)\in\Complex$ and that the domain~$F(h(\UD))$ has angular derivative at~$+\infty$ in the sense of Rodin\,--\,Warschawski. Moreover, to simplify the further argument, note that condition~\eqref{EQ_Bertilsson-condition} is invariant under locally conformal change of variables. Therefore, applying~\cite[Proposition~2.1]{Pavel} we may assume that the DW-point of~$(\phi_{s,t})$ is at~$1$. Then the ray $S:=\{h(x)-\xi\colon \xi\ge0\}$ lies in~$\Complex\setminus h(\UD)$ and hence $F(h(\UD))\subset\Pi:=\{z\colon|\Im z|<1/2\}$.

For $k\in\Natural$ denote $\Delta_k:=\big[\gamma k/(2\pi),\gamma(k+1)/(2\pi)\big]$. Further, for $u\in\Real$ denote by  $J(u)$  the connected component of $F(h(\UD))\cap\{\zeta\colon \Re \zeta=u\}$ that intersects~$\Real$ and let $v''(u):=\sup\{\Im \zeta\colon\zeta\in J(u)\}$, $v'(u):=\inf\{\Im \zeta\colon\zeta\in J(u)\}$. Then
\begin{equation}\label{EQ_alpha_k}
\frac{\alpha_k}{2\pi}=1-\Big(\inf_{u\in\Delta_k} v''(u)\,-\,\sup_{u\in\Delta_k} v'(u)\Big)\quad\text{for all~$k\in\Natural$}.
\end{equation}
Using the fact that $h(\UD)$ is invariant w.r.t. the translations $w\mapsto w+t$, $t\ge0$, we will now prove the following claim.

\Claim{C}{For all $k\in\Natural$ large enough, $\mathsf{area}\, \big\{\zeta\in\Pi\colon \Re\zeta\in\Delta_k,\,\zeta\not\in J(\Re\zeta)\big\}\ge M\alpha_{k+1}$, where $M>0$ is a constant depending only on~$\gamma$ and $\mathsf{area}\,\cdot$ stands for the two-dimensional Lebesgue measure} Since $F(h(\UD))$ has angular derivative at~$+\infty$ in the sense of Rodin\,--\,Warschawski, \eqref{EQ_alpha_k} implies that $\alpha_k\to0$ as $k\to+\infty$. Fix $k_0\in\Natural$ such that $\alpha_k<\pi/2$ for all $k>k_0$.

Now from the fact that $\{w-\xi\colon\xi\ge0\}\subset\Complex\setminus h(\UD)$ for any $w\in\Complex\setminus h(\UD)$ it follows that if $k\ge k_0$ and $\zeta\in\Pi\setminus F(h(\UD))$ with $\Re\zeta\in\Delta_{k+1}$, then
$$
\zeta\in\Big\{u+iv\colon \tfrac14<|v|<\tfrac12,\, u\le \Re\zeta,\, \sin(2\pi v)=e^{2\pi(u-\Re\zeta)}\sin(2\pi\zeta)\Big\}\subset\Pi\setminus F(h(\UD)).
$$
The statement of Claim C follows now by elementary computation from~\eqref{EQ_alpha_k}.
\vskip2mm

By \cite[Theorem~1]{RodinWarschawski}, $\mathsf{area}\,\big(\{\zeta\colon\Re\zeta\ge0\}\cap\Pi\setminus\cup_{u\ge0}J(u)\big)<+\infty$. Therefore, by Claim~C the series $\sum\alpha_k$ converges.
\end{proof}

\end{document}